\documentclass[11pt, reqno]{amsart}
\usepackage{amssymb, amscd, mathrsfs, wasysym} 
\usepackage{xcolor}
\usepackage{tikz}
\providecommand{\coloneqq}{\mathrel{:=}}
\newcommand{\Ocal}{\mathcal O}
\numberwithin{equation}{section}
\numberwithin{figure}{section}
\theoremstyle{plain}
\newtheorem{theorem}{Theorem}[section]
\newtheorem{corollary}[theorem]{Corollary}
\newtheorem{proposition}[theorem]{Proposition}
\newtheorem{lemma}[theorem]{Lemma}
\theoremstyle{definition}
\newtheorem{definition}[theorem]{Definition}
\newtheorem{example}[theorem]{Example}
\theoremstyle{remark}

\usepackage[a4paper,left=22mm,right=22mm,top=22mm,bottom=24mm,
  headheight=12pt,headsep=14pt]{geometry}
\newcommand \bei {\begin{itemize}}
\newcommand \eei {\end{itemize}}
\newcommand\bel[1]{\begin{equation}\label{#1}}
\newcommand\be {\begin{equation}}
\newcommand\ee{\end{equation}}
\newcommand \Sec {\text{Sec}} 
\newcommand \Dscr {\mathscr D}
\newcommand \Lscr {\mathscr L}
\newcommand \Rscr {\mathscr R}
\newcommand \Ascr {\mathscr A}
\newcommand \Cscr {\mathscr C}
\newcommand \Wscr {\mathscr W}
\newcommand \Wscrt {\mathscr W} 

\newcommand \Fscr {\mathscr F}

\newcommand \Hcal {\mathcal H}
\newcommand \Lcal {\mathcal L}
\newcommand \Acal {\mathcal A}
\newcommand \SigmaMCut {\Sigma^{\setminus\text{Cut}}}
\newcommand \SigmaCCut {\Sigma^{\cap\text{Cut}}}

\newcommand \la 		\langle
\newcommand \ra 		\rangle

\newcommand \del 	\partial
\newcommand\eps{\ensuremath{\varepsilon}}
\newcommand\RR {\ensuremath{\mathbb{R}}}

\newcommand\II{\ensuremath{\mathrm{II}}}
\newcommand\sn{\ensuremath{\mathrm{sn}}}
\newcommand\md{\ensuremath{\mathrm{md}}}

\DeclareMathOperator{\Cut}{Cut}
\DeclareMathOperator{\Hess}{Hess}

\DeclareMathOperator{\diam}{diam}
\DeclareMathOperator{\grad}{grad}
\DeclareMathOperator{\dv}{div}
\DeclareMathOperator{\dive}{div}
\newcommand \dist {\mbox{dist}}
\DeclareMathOperator{\tr}{tr}
\DeclareMathOperator{\inj}{inj}
\let\oldmarginpar\marginpar
\renewcommand\marginpar[1]{\-\oldmarginpar[\raggedleft\footnotesize #1]%
{\raggedright\footnotesize #1}}
\newcommand \Wscrtt {\widetilde \Wscr}

\begin{document}

\title[] 
{Willmore inequalities beyond the injectivity radius}

\author[Philippe G. L{\smaller E}Floch]{Philippe G. L{\smaller E}Floch}
\thanks{Laboratoire Jacques-Louis Lions \& Centre national de la recherche scientifique, Sorbonne University, 4 Place Jussieu, 75252 Paris, France.  Email: \texttt{contact@philippelefloch.org}.}

\date{} 

\begin{abstract} We generalize a method first introduced by Li, Yau, Simon, and Topping (for closed surfaces immersed in the flat Euclidean space), and we establish a global-in-nature estimate, with an explicit cut-locus contribution, on the Willmore energy of closed surfaces that are immersed in a complete Riemannian three-manifold and need not be embedded. Our estimate applies to ambient spaces whose sectional curvature is bounded above and whose injectivity radius may be {\sl finite.} For positive upper curvature bounds, the global weighted estimates require an explicit radius restriction. The Willmore energy is bounded below by a functional involving the self-intersection number of the surface and the injectivity radius of the ambient space. Furthermore, by investigating in what sense our inequalities are optimal and realized by  geodesic spheres in constant curvature spaces, we are led to introduce a weighted version of the Willmore functional, which is adapted to geodesic spheres in the model spaces and takes the sectional curvature of the ambient space into account. As a corollary, we obtain  sufficient conditions for an immersed surface in a curved space to be embedded.  
\end{abstract}

\maketitle
  

\section{Introduction and main results}
\label{sec:intro}

\subsection*{Objective and background}

We consider closed surfaces that are immersed in a complete Riemannian three-manifold, but need not be embedded. Our objective is to establish \emph{global-in-nature lower bounds} on their Willmore energy. 
In particular, we arrive at sufficient conditions for a surface in curved space to be embedded. The Euclidean lower bound can fail when the ambient injectivity radius is finite: a flat three-dimensional quotient contains an embedded totally geodesic two-torus with zero Willmore energy. The cut-locus contribution quantifies this obstruction. The resulting estimates give conditions for the absence of self-intersections and upper bounds on the ambient injectivity radius when the energy is small relative to the multiplicity. 

The method presented here allows us to consider an ambient space with {\sl finite} injectivity radius, and to encompass surfaces ``going beyond'' the injectivity radius of the ambient manifold. We generalize a method proposed by Li and Yau \cite{LiYau:Willmore}, Simon~\cite{Simon}, and Topping~\cite{Topping:MCF}, who were interested in surfaces immersed in {\sl flat} space with, therefore, {\sl infinite} injectivity radius. 

Working in a curved space with possibly finite injectivity radius leads us to study the cut-locus contributions of the Willmore energy. In particular, our analysis leads us to the functionals $\Fscr_n, \Ascr, \Lscr$, defined later in this section, which are naturally associated with a surface immersed in a curved space with finite injectivity radius. We focus on the explicit cut-locus contribution; in the absence of a cut contribution, the nonpositive-curvature bound recovers known Willmore inequalities~\cite{Ritore,Scharrer}.

The {\sl Willmore energy of a surface} $\Sigma \hookrightarrow \RR^3$ immersed in the Euclidean space is defined 
by 
\bel{eq:180}
\Wscr(\Sigma) \coloneqq \int_\Sigma H^2 \, d\Acal_\Sigma. 
\ee
Here, given an immersion $\varphi \colon \Sigma \to \RR^3$ of a closed, oriented $2$-surface, we denote by $H$ its (scalar) mean curvature, by $g_\Sigma$ the induced metric, and by $d\Acal_\Sigma$ the volume form on the surface. Of course, the energy depends on the immersion, but we often identify the surface with its image $\Sigma = \varphi(\Sigma)\subset \RR^3$. 

We also introduce the functional $\Wscr_0(\Sigma) \coloneqq \int_\Sigma \left( H^2 - \kappa \right) d\Acal_\Sigma$, where $\kappa$ denotes the Gauss curvature of the surface. Then, in terms of the principal curvatures $k_1, k_2$ of the surface, we have $2 \, H = k_1 + k_2$ and $\kappa = k_1 \, k_2$, so that
\be
\Wscr_0(\Sigma) = {1 \over 4} \int_\Sigma \left( k_1 - k_2 \right)^2 d\Acal_\Sigma \, \geq 0,
\qquad \Sigma \subset \RR^3.
\ee
Thanks to the Gauss-Bonnet theorem, these two functionals are related as 
\be
\Wscr(\Sigma) = \Wscr_0(\Sigma) + 2 \pi \, \chi(\Sigma),  \qquad \Sigma \subset \RR^3,
\ee
where $\chi(\Sigma)= 2(1-h)$ (with $h=0, 1, \ldots$) is the Euler characteristic  of the surface. The inequality
$\Wscr(\Sigma) \ge 2 \pi \, \chi(\Sigma)$ then provides one with an obvious, yet non-trivial, lower bound for the Willmore energy, provided the surface is diffeomorphic to a topological sphere, namely
\bel{eq:181}
\Wscr(\Sigma) \ge 4 \pi, \qquad \Sigma \simeq S^2 \subset \RR^3.
\ee
In fact, it is known (see, for instance \cite[p.~362]{Topping:W}) that \eqref{eq:181} holds for surfaces of {\sl any} genus.

A much stronger estimate is established in Li and Yau~\cite[Theorem~6]{LiYau:Willmore} (see, also, Topping~\cite{Topping:MCF}), which
holds for {\sl any} closed surface $\Sigma$ (independently of its topology): 
\bel{eq:160}
\Wscr(\Sigma) \ge 4 \pi k(\Sigma),   \qquad \Sigma \subset \RR^3.
\ee
Here, $k(\Sigma)$ denotes the largest number of self-intersections of the surface, i.e., by definition, there exists $p \in \RR^3$ and $k$ distinct points $p_1, \ldots, p_k \in \Sigma$ such that $\varphi(p_1) = \ldots = \varphi(p_k) = p$.
In particular, if the Willmore energy is less than $8\pi$ then the surface admits no self-intersection and the map $\varphi \colon \Sigma \to \RR^3$ is actually an {\sl embedding.} For small Willmore spheres in Riemannian three-manifolds, see~\cite{LM}.  

Our primary objective in the present paper is to generalize the Li--Yau estimate~\eqref{eq:160} in order to encompass surfaces immersed in a {\sl curved ambient manifold} with (possibly) {\sl finite injectivity radius.}  We will also treat properly immersed non-compact surfaces whose ends are assumed to be asymptotically cylindrical.


\subsection*{Notation}

We introduce the notation used throughout this article. We assume given a complete Riemannian $3$-manifold (without boundary) $(M, g)$ whose inner product is denoted by $g(\cdot, \cdot)$ or $\la \cdot, \cdot \ra$, whose Levi-Civita connection is denoted by $\nabla$. We consider oriented $2$-manifolds $\Sigma$ and immersions $\varphi \colon \Sigma \to M$ in the ambient space $M$ and,
as mentioned earlier, we identify $\Sigma$ with its image $\varphi(\Sigma) \subset M$. We denote by $g_\Sigma \coloneqq \varphi^* g$ the metric induced on the surface from the ambient metric $g$, and $d\Acal_\Sigma$ denotes the corresponding volume form on $\Sigma$.
Next, we denote by $n \in TM$ a locally chosen unit normal vector to $\Sigma\subset M$, and by $\II$ the second fundamental form of the surface, with the convention that for all tangent vector fields $X, Y$  
\be
\aligned
\II(X, Y) \coloneqq - g(\nabla_X Y,n) = g(\nabla_X n,Y).
\endaligned
\ee
The mean curvature $H$ is then defined from the trace of the second fundamental form by 
$
H \coloneqq \tfrac{1}{2} \tr (\II). 
$
The densities used below are invariant under simultaneous reversal of $n$ and $H$. This notation allows us to introduce the {\sl Willmore energy} of the immersion $\varphi \colon  \Sigma \to M$ (now in curved space) by the {\sl same} formula\footnote{Observe that  this definition makes sense if $\Sigma$ is closed or, more generally, has finitely many cylindrical ends (cf.~the discussion at the end of Section~\ref{sec:2}, below).}
 \eqref{eq:180}. 

For $p\in\varphi(\Sigma)$, the multiplicity is $k_p(\Sigma)\coloneqq\#\varphi^{-1}(p)$. By taking the maximum over the image, we set
\be
k(\Sigma) \coloneqq \max_{p\in\varphi(\Sigma)} k_p(\Sigma).
\ee
Furthermore, we consider the ambient space $M$ and denote by $B_R(p) \subset M$ the open geodesic ball centered at some point $p$ and with radius $R>0$. We  denote by $\dist_p \colon M \to [0, \infty)$ the {\sl distance function} from this point (defined with the ambient metric) and by $\inj_p(M)$ and $\Cut_p(M)$ the {\sl injectivity radius} and {\sl cut locus} of $p$, respectively.  By taking the infimum over $p$, we define the injectivity radius along the immersion as 
\be
\inj(M) \coloneqq \inf_{p \in \Sigma} \inj_p(M).
\ee


All area and length integrals are taken on the parameter surface with its induced metric, so distinct sheets are counted separately. In particular, $|\Sigma\cap\Cut_p(M)|$ means the length of $\varphi^{-1}(\Cut_p(M))$. Exact decompositions across the cut locus require the finite-length hypothesis in the equation~\eqref{cutSigma}, below. An estimate with infinite cut length is a trivial extended-real inequality.

\subsection*{Non-positively curved spaces}

We are going to derive quantitative bounds on the Willmore energy in terms of geometrical quantities such as the self-intersection number of the surface $k(\Sigma)$ and the injectivity radius of the ambient manifold $\inj(M)$. Our most general estimate will be stated only later in Section~\ref{sec:WF} and only requires that the sectional curvature of the ambient manifold is bounded above. In this introductory section, we find it convenient to state our result under the stronger assumption that the sectional curvature is non-positive. (We refer to Theorem~\ref{theo:46} below for a more general statement.) 

\begin{theorem}[Willmore functional for non-positively curved spaces. I]
\label{thm:Will1}
Let $(M, g)$ be a complete Riemannian $3$-manifold with non-positive sectional curvature
and $\Sigma \hookrightarrow M$ a closed, oriented immersed surface. Then, the Willmore energy of $\Sigma$ satisfies the lower bound
\bel{eq:190-infi}
{1 \over 4 \pi} \Wscr(\Sigma) \geq k(\Sigma)
 - \inf_{\substack{0<r<d\\r\ {\rm regular}}} {|\Sigma \cap S_{r}(p)| \over 2 \pi \, r}
\ee
with $d\coloneqq\dist_p( \Sigma \cap \Cut_p(M))$, 
where $p$ is a point achieving $k(\Sigma)=k_p(\Sigma)$, 
$S_{r}(p)$ is the foliation by geodesic spheres centered at $p$,
and $|\Sigma\cap S_r(p)|$ denotes the length of the intersection. The infimum is over regular values of $\dist_p\circ\varphi$, and $d=+\infty$ if the cut set is empty.
\end{theorem}

Under the assumptions of Theorem~\ref{thm:Will1}, we immediately deduce by the cutoff calculation in Section~\ref{sec:33}, below, that the Willmore energy of $\Sigma$ satisfies the lower bound
\bel{eq:190}
{1 \over 4 \pi} \Wscr(\Sigma) \geq k(\Sigma) - \inf_{n>0} \Fscr_n(\Sigma),
\ee
where for finite $\inj(M)$ we set
\be
\Fscr_n(\Sigma) \coloneqq {n \over 2 \pi \big( \inj(M) \big)^2} \max_{p \in \Sigma} \int_\Sigma \Bigg({\dist_p \over \inj(M)} \Bigg)^{n-2} \, d\Acal_\Sigma.
\ee
If $\inj(M)=+\infty$, the cutoff limit $R\to+\infty$ gives the bound with $\Fscr_n(\Sigma)=0$. In particular, taking $n=1$ and $n=2$, respectively, we get the geometrically significant quantities, which we refer to as the {\sl scaled area functional}
\be
\Ascr(\Sigma) \coloneqq \Fscr_2(\Sigma) = {1 \over \pi \inj(M)^2} \int_\Sigma d\Acal_\Sigma,
\ee
and the {\sl scaled length functional}
\be
\Lscr(\Sigma) \coloneqq \Fscr_1(\Sigma) = {1 \over 2\pi \inj(M)} \, \max_{p \in \Sigma} \int_\Sigma \frac{d\Acal_\Sigma}{\dist_p}.
\ee
Then, from \eqref{eq:190}, we deduce the (weaker) inequality
\bel{eq:191}
{1 \over 4 \pi} \Wscr(\Sigma) \geq k(\Sigma) - \min \big( \Lscr(\Sigma), \Ascr(\Sigma)\big).
\ee

The estimate in the equation~\eqref{eq:191} is stronger than the elementary bound $\Wscr(\Sigma)\geq0$ whenever $\min(\Lscr(\Sigma),\Ascr(\Sigma))<k(\Sigma)$. It also gives an explicit sufficient condition for embeddedness: if
\be
\frac{\Wscr(\Sigma)}{4\pi}+\min\big(\Lscr(\Sigma),\Ascr(\Sigma)\big)<2,
\ee
then $k(\Sigma)<2$. Since the multiplicity is a positive integer, the immersion is injective and, by compactness, is an embedding.

The functionals $\Fscr_n$ are dimensionless and are intended to ``count'' the number of times the surface $\Sigma$ is ``going beyond'' the injectivity radius 
of the ambient space. For instance, concerning the area $\Ascr(\Sigma)$, we may consider the example where $M$ is a flat torus $[0,a]^2\times[0, b]$ while $\Sigma \subset M$ is a round sphere with radius $c$ so that, provided $a < b$ and $c < b$, the surface $\Sigma$ is ``wrapped'' around $M$, with $\inj(M) = \tfrac{a}{2}$, and we find $\Ascr(\Sigma) = 16 (c/a)^2$, which involves the ``counting ratio'' $c/a$.

A second main result, also stated now, will be established in this paper and provides us with a lower bound, expressed as a function of the self-intersection number of the surface and the length of the intersection of the surface with the cut locus from $p$.

\begin{theorem}[Willmore functional for non-positively curved spaces. II]
\label{theo:34-e}
Let $\Sigma \hookrightarrow M$ be a closed surface 
immersed in a non-positively curved space $(M,g)$.
If $p \in \Sigma$ is a $k$-fold point of this surface, then one has
\bel{Simon5-ee}
{1 \over 2} \Wscr(\Sigma) \geq
 2 \pi k_p(\Sigma) 
 - 2 \, {|\Sigma\cap\Cut_p(M)|\over\inj(M)}. 
\ee  
\end{theorem}

Furthermore, in this paper we also propose a generalization of the above results to ambient spaces with {\sl curvature bounded above.} 
First, we observe that for ambient spaces with curvature bounded above, there are analogues of Theorems~\ref{thm:Will1} and \ref{theo:34-e}, which, however, are not very natural if a sharp bound is sought. We are thus led to propose a new functional of the form (cf.~Section~\ref{sec:WF} for the notation) 
\be
\Wscrt_{K,p} (\Sigma) \coloneqq \int_\Sigma {H^2 \over \sn_K'(\dist_p)}\, d\Acal_\Sigma,  
\ee
which is defined for $\max_\Sigma\dist_p<\pi/(2\sqrt K)$ when $K>0$ and is referred to as the {\sl $K$-weighted Willmore functional}. Our functional depends on the upper bound of the sectional curvature, denoted by $K$, and, by relying on a comparison argument for manifolds with curvature bounded above, can be estimated below. Cf.~Theorems~\ref{theo:34-ff} and ~\ref{theo:46}.


\subsection*{Outline of this paper}

In Section~\ref{sec:2}, we generalize the Simon identity~\cite{Simon} and we explain how to distinguish between contributions due to the ambient curvature, the cut locus, etc. Our identity therein involves an arbitrary test-function which is specified in the following sections. In Section~\ref{sec:3}, we introduce a decomposition of the Willmore energy and we give a proof of Theorem~\ref{theo:34-e}. In Section~\ref{sec:33}, we establish Theorem~\ref{thm:Will1} by making a suitable choice of the test-function arising in Section~\ref{sec:2}. Next, in Sections~\ref{sec:4} and \ref{sec:WF}, we study surfaces in an ambient space with curvature bounded above and, in Theorems~\ref{theo:34-ff} and ~\ref{theo:46}, we provide a generalization of the theorems stated in this introduction.


\section{An identity for immersed surfaces in curved space}
\label{sec:2}

Our first task is to generalize an identity established by Simon~\cite{Simon} and Topping~\cite{Topping:MCF} for surfaces in the Euclidean space. Relying on the notation and assumptions in Section~1, we thus consider a closed surface $\Sigma$ immersed in a Riemannian space $(M,g)$. 

We write $u\coloneqq\dist_p\circ\varphi$, $Z_p\coloneqq\varphi^{-1}(p)$, and $E_p\coloneqq\varphi^{-1}(\Cut_p(M))$. Given a point $p \in \Sigma$, we consider the distance function $r\coloneqq\dist_p$ from $p$ together with the unit vector field $\del_r \coloneqq \grad_g r$. This vector field is tangent to the radial geodesics from $p$ and is well-defined and smooth {\sl except} at the origin $p$ and on the cut locus $\Cut_p(M)$. Throughout, we assume that the intersection of the cut locus of $p$ with the surface $\Sigma$ has finite length on the parameter surface: 
\bel{cutSigma}
\Hcal^1_{g_\Sigma}(E_p)<+\infty.
\ee
This hypothesis implies that $E_p$ has zero surface area. It is required for the exact global cut identities, but not for local estimates supported away from the cut set. Consequently, an integral of an integrable function over $\Sigma$ can be replaced by an integral over $\Sigma \setminus \Cut_p(M)$. 
Furthermore, $N$ denotes the unit conormal in the parameter surface pointing from the minus side to the plus side of the rectifiable jump set. The symbols $(\del_r^\pm)^\parallel$ mean the two approximate traces of $\nabla_\Sigma u$ on that surface. By Lemma~\ref{cut-measure-lemma}, proved in the appendix,
\be
(\del_r^- -\del_r^+)^\parallel=\beta_pN,\qquad
\nu_p\coloneqq\beta_p\Hcal^1_{g_\Sigma}\restriction E_p,\qquad0\leq\beta_p\leq2.
\ee
The density is extended by zero off the jump set. Thus, with the sign fixed by the jump convention above, cut integrals in this article are intrinsic measure integrals of the form $\int_{E_p}\psi(u)\,d\nu_p$, where $\psi$ is its radial weight. It does not require a smooth boundary for $\Sigma\setminus\Cut_p(M)$ or the existence of two ambient radial limits at every cut point.

For locally Lipschitz radial functions, we set the derivative to zero at points of nondifferentiability. This convention is compatible with the tangential chain rule, since $\nabla_\Sigma u=0$ almost everywhere on the inverse image of a null set of radial values.

We need some more notation. In particular, the definitions below apply to the vector field $\psi \, \del_r$, where $\psi = \psi(r)$ is an arbitrary function of the  distance function $r=\dist_p$ from $p$. Given a Lipschitz continuous vector field $X$ defined on a neighborhood of the surface ${\Sigma \setminus \Cut_p(M)}$ and given a point $q$ in ${\Sigma \setminus \Cut_p(M)}$,
 we introduce the orthogonal projection 
$
X^\parallel \coloneqq X - \la X, n \ra n
$
 onto the tangent space $T_q \Sigma$. (Recall that $n$ denotes the unit normal to $\Sigma$.) By letting $\{ e_1, e_2 \}$ be any orthonormal frame for the metric $g_\Sigma$ (which exists at least locally in the manifold), the divergence of the vector field under consideration is defined as
\bel{eq:dive}
\dive_\Sigma X 
\coloneqq\la \nabla_{e_1} X, e_1 \ra + \la \nabla_{e_2} X, e_2 \ra
= \dv_g X - \la \nabla_n X, n \ra
\ee
and is well-defined in $\Sigma \setminus \Cut_p(M)$.

We are now in a position to state the following result.

\begin{proposition}[An identity for immersed surfaces in curved space]
\label{prop:fundamental}
Let $\Sigma$ be a closed, smoothly immersed surface in a Riemannian space $(M,g)$ satisfying the equation~\eqref{cutSigma}. 
Given a point $p \in \Sigma$ of $k$-fold self-intersection of this surface and any locally Lipschitz function $\psi\colon\RR^+\to\RR$ for which the densities below are integrable and with finite limit
\bel{eq:phi}
\psi_0\coloneqq\lim_{s \to 0} s \, \psi(s), 
\ee
one has
\bel{Simon2}
\aligned
& 2 \pi \, \psi_0 \, k_p 
  = \int_\Sigma \delta_p\, d\Acal_\Sigma
- \int_{\Sigma \setminus \Cut_p(M)} \hskip-.15cm \rho_p \, d\Acal_\Sigma
- \int_{{\Sigma \cap \Cut_p(M)}} \hskip-.15cm \psi (r)\, \la \del_r^+ - \del_r^-, N \ra \, d\Lcal_\Sigma,
\endaligned
\ee
where $d\Lcal_\Sigma$ denotes the induced volume form on $\Sigma \cap \Cut_p(M)$ and we have set 
\bel{eq:defchitheta}
\aligned
{1 \over 2} \delta_p \coloneqq \, &  \psi(r) \, \la n, \del_r \ra \, H + \psi'(r) \, \la n, \del_r \ra^2,
\\
\rho_p \coloneqq \, & \tr_{g_\Sigma} \Big( \psi(r) \Hess r + \psi'(r) ( g - d r \otimes dr ) \Big).
\endaligned
\ee
\end{proposition}

Observe that the densities are evaluated almost everywhere off $E_p\cup Z_p$. For the radial fields used below, they are integrable; at $Z_p$ this follows from normal coordinates and away from $Z_p$ from Lemma~\ref{cut-measure-lemma}, below. We now proceed with the proof of Proposition~\ref{prop:fundamental}. 
Given a sufficiently small real $\sigma>0$, we remove from the surface the geodesic ball $B_\sigma(p)$ centered at the point $p$, by defining
\be
\aligned
& \SigmaMCut_\sigma \coloneqq \Big( \Sigma \setminus \Cut_p(M) \Big) \setminus B_\sigma(p),
\qquad
 \SigmaCCut_\sigma \coloneqq \Big( \Sigma \cap \Cut_p(M) \Big) \setminus B_\sigma(p).
\endaligned
\ee
We are going to integrate (the divergence of a vector field) on the subset ${\SigmaMCut_\sigma}$, whose boundary terms consist of the inner-circle fluxes and the jump measure on the cut set
\be
\del \SigmaMCut_\sigma =  \Big( \Sigma \cap S_\sigma(p) \Big) \cup \Big( \Sigma \cap \Cut_p(M) \Big),
\ee
where $S_\sigma(p)$ denotes the sphere centered at $p$ of radius $\sigma$
and, since $\sigma$ is sufficiently small, these two components are disjoint.

\begin{lemma}[Divergence formula for vector fields]
\label{IntegrationLemma}
Under the conditions in Proposition~\ref{prop:fundamental} and provided the tangential projection of $X$ has locally bounded variation away from $Z_p$, integrable regular derivatives, integrable normal component, and only jump derivatives supported on $E_p$, one has (with the notation already introduced)
\be
\aligned
& \int_{\SigmaMCut_\sigma} \left( \dv_\Sigma X\right) d\Acal_\Sigma
\\
& =
2 \int_{\SigmaMCut_\sigma} \la n, X \ra \, H \, d\Acal_\Sigma
-
\int_{{\Sigma \cap \Cut_p(M)}} \la N, X^\parallel_+ - X^\parallel_- \ra \, d\Lcal_\Sigma
+
\int_{\Sigma \cap S_\sigma(p)} \la N, X^\parallel \ra \, d\Lcal_\Sigma,
\endaligned
\ee
where, in the latter integral, $d\Lcal_\Sigma$ stands for the induced volume form on $\Sigma \cap S_\sigma(p)$. 
\end{lemma}

\begin{proof} This is the standard first variation formula,  
which is established by writing (in view of~\eqref{eq:dive} and on the set ${\SigmaMCut_\sigma}$) 
\be
\begin{aligned}
\dive_\Sigma X&=
\dv_\Sigma X^\parallel + \sum_{i=1}^2 \la \nabla_{e_i} X^\perp, e_i \ra
= \dv_\Sigma X^\parallel - \sum_{i=1}^2 \la X, \left( \nabla_{e_i} e_i \right)^\perp \ra
\\
&= \dv_\Sigma X^\parallel + \sum_{i=1}^2 \la X,n\ra\,\II(e_i, e_i)
= \dv_\Sigma X^\parallel + 2 \la n, X \ra \, H,
\end{aligned}
\ee
integrating this identity on the surface with the small disks removed, and then applying the Gauss--Green formula with its jump term, justified by Lemma~\ref{cut-measure-lemma}, below.
\end{proof}

\begin{lemma}[Divergence of the radial vector field]
\label{200}
Under the conditions in Proposition~\ref{prop:fundamental} and for an arbitrary function $\psi \colon \RR^+ \to \RR$, one has, 
with $\rho_p$ defined in \eqref{eq:defchitheta}, 
\be
\dive_\Sigma \left( \psi(r) \, \del_r \right)
= - 2 \psi'(r) \, \la n, \del_r \ra^2 + \rho_p \quad 
\text{ almost everywhere on }\Sigma\setminus\Big(\Cut_p(M)\cup\{p\}\Big).
\ee
\end{lemma}

\begin{proof}
The classical divergence, at smooth points, is
\be
\dive_\Sigma(\psi\del_r)=\psi\tr_{g_\Sigma}\Hess r
+\psi'(1-\langle n,\del_r\rangle^2)
=\rho_p-2\psi'\langle n,\del_r\rangle^2. \qedhere 
\ee
\end{proof}

\begin{proof}[Proof of Proposition~\ref{prop:fundamental}]  
Applying Lemma~\ref{IntegrationLemma} to the vector field $\psi(r) \, \del_r$ with an arbitrary function $\psi \colon \RR^+ \to \RR$, and re-arranging the terms, we arrive at
\be
\begin{aligned}
& \int_{\SigmaMCut_\sigma} \dive_\Sigma \left( \psi(r) \, \del_r \right) d\Acal_\Sigma
-
2 \int_{\SigmaMCut_\sigma} \psi(r) \, \la n, \del_r \ra \, H\, d\Acal_\Sigma
\\
&=
- \int_{{\Sigma \cap \Cut_p(M)}} \psi(r) \, \la N, \del_r^+ - \del_r^- \ra \, d\Lcal_\Sigma
+
\int_{\Sigma \cap S_\sigma(p)} \psi(r) \, \la N, \del_r \ra \, d\Lcal_\Sigma.
\end{aligned}
\ee
Next, in view of Lemma~\ref{200} and with notation $\delta_p$ in \eqref{eq:defchitheta}, we may rewrite this identity in the form
\bel{Simon2a}
\aligned
&\int_{\SigmaMCut_\sigma} \big( - \delta_p + \rho_p \big) \, d\Acal_\Sigma
\\
& =
- \int_{{\Sigma \cap \Cut_p(M)}} \psi \la \del_r^+ - \del_r^-, N \ra \, d\Lcal_\Sigma
+
\int_{\Sigma \cap S_\sigma(p)} \psi \la \del_r, N\ra \, d\Lcal_\Sigma.
\endaligned
\ee

We now take the limit as $\sigma \to 0$ in the equality~\eqref{Simon2a}. Let $p$ be a self-intersection point covered $k_p$ times in the immersion $\Sigma \to M$, and observe that, for points on $S_\sigma(p)$, $\la \del_r, N\ra \to \la \del_r, - \del_r \ra = -1$ as $\sigma \to 0$. Provided $\lim_{r \to 0} r \psi(r) = \psi_0$,
we can take the limit as $\sigma \to 0$ of the second term on the right-hand-side of~\eqref{Simon2a} and obtain
\be
\lim_{\sigma\to0}\int_{\Sigma\cap S_\sigma(p)}\psi\la\del_r,N\ra\,d\Lcal_\Sigma
=
\lim_{\sigma\to0} \sum_{1}^k \psi_0 {1 \over \sigma} \left( - 2 \pi \sigma \right)
=
- 2 \pi \psi_0 k_p.
\ee
Hence, as $\sigma \to 0$, \eqref{Simon2a} gives
\bel{eq:identity-closed}
\int_{\Sigma \setminus \Cut_p(M)} \big( - \delta_p + \rho_p \big) \, d\Acal_\Sigma
=
- \int_{{\Sigma \cap \Cut_p(M)}} \psi \la \del_r^+ - \del_r^-, N \ra \, d\Lcal_\Sigma
- 2 \pi \psi_0 k_p,
\ee
in which the integration is performed away from the cut locus so that all the integrands are smooth functions except for an (integrable) singularity at the origin $p$ where $r$ vanishes.
The proof of Proposition~\ref{prop:fundamental} is completed.
\end{proof}

 
\subsection*{Remarks on non-closed surfaces}
  
The identity extends under the integrability and radial-flux conditions stated below to surfaces $\Sigma$ that are not closed but admit cylindrical ends, defined as follows (which is a sufficient broad notion in order to apply, for instance, to unit-radius cylinders). 

\begin{definition}
A properly immersed surface $\Sigma$ admits $k_\infty$ \emph{cylindrical ends} if there exists $R>0$ and $p_0 \in \Sigma$ such that $\Sigma \setminus B_R(p_0)$ is the union of $k_\infty$ ``ends'', each being diffeomorphic to the cylinder $[1, +\infty) \times S^1$ and with a metric $g_\Sigma$ whose coefficients converge uniformly in $y$ to $dr_0^2+dy^2$ where $r_0\coloneqq\dist_{p_0}$ and $y \in [0, 2\pi]$ is a coordinate function on the $1$-torus $S^1$.
\end{definition}
 
For the chosen point $p$, assume that the cylindrical coordinates above use $p_0=p$, that the cut set is contained in a compact part and has finite length, and that $\delta_p,\rho_p\in L^1(\Sigma)$. We can revisit the calculation above by assuming that the test-function satisfies
\bel{eq:phi-withEnds}
\psi_0\coloneqq\lim_{s \to 0} s \, \psi(s), \qquad 
\psi_\infty\coloneqq\lim_{s \to +\infty} \psi(s).
\ee
The identity \eqref{Simon2} becomes
 \bel{Simon2-withends}
\aligned
& 2 \pi \big( \psi_0 \, k_p - \psi_\infty \, k_\infty \big)
\\
& = \int_\Sigma \delta_p\, d\Acal_\Sigma
- \int_{\Sigma \setminus \Cut_p(M)} \hskip-.15cm \rho_p \, d\Acal_\Sigma
- \int_{{\Sigma \cap \Cut_p(M)}} \hskip-.15cm \psi (r)\, \la \del_r^+ - \del_r^-, N \ra \, d\Lcal_\Sigma.
\endaligned
\ee
Hence, in order to extend our lower bounds to non-compact surfaces, it suffices to add the contribution $-2\pi\psi_\infty k_\infty$ to the self-intersection contribution $2\pi\psi_0 k_p$. 

To derive \eqref{Simon2-withends}, we integrate first within a large geodesic ball $B_R(p)$ and \eqref{eq:identity-closed} reads (for all sufficiently large $R>0$)
\bel{eq:identity-nonclosed}
\aligned
&  2 \pi \psi_0 k_p
+ \int_{\big( \Sigma \setminus \Cut_p(M) \big) \cap B_R(p)} \big( - \delta_p + \rho_p \big) \, d\Acal_\Sigma
\\
&= - \int_{\big(\Sigma \cap \Cut_p(M) \big) \cap B_R(p)} \psi \la \del_r^+ - \del_r^-, N \ra \, d\Lcal_\Sigma
 + \int_{S_R(p) \cap \Sigma} \psi(r) \la \del_r, N\ra \, d\Lcal_\Sigma,
\endaligned
\ee
where, in the latter integral, $N$ denotes the outward normal unit vector field along the boundary $S_R(p) \cap \Sigma$ of $B_R(p) \cap \Sigma$.
Recalling that $\Sigma$ admits $k_\infty$ cylindrical ends, for which the radial vector field $\del_r$ asymptotically becomes aligned with the normal to the geodesic sphere $B_R(p)$, the metric convergence gives $|\nabla_\Sigma u|\sqrt{(g_\Sigma)_{yy}}=(g_\Sigma)_{yy}/\sqrt{\det g_\Sigma}\to1$ on each end, uniformly in $y$. Thus we find
\be
\lim_{R\to+\infty} \int_{S_R(p) \cap \Sigma} \psi(r) \la \del_r, N\ra \, d\Lcal_\Sigma
= k_\infty(\Sigma) 2 \pi \, \psi_\infty.
\ee
Finally, by taking the limit $R \to +\infty$ in \eqref{eq:identity-nonclosed}, using the stated $L^1$ hypotheses, we arrive at \eqref{Simon2-withends}. For the field $\psi(r)=1/r$, one has $\psi_\infty=0$ and the end contribution vanishes.


\section{A lower bound involving the cut locus} 
\label{sec:3}

\subsection*{Decomposition of the Willmore energy}

As will become clear below, the decomposition to be derived now is most relevant for flat or negatively curved manifolds, and indeed we will first focus on such manifolds. In Proposition~\ref{prop:fundamental}, we are going to choose the test-function $\psi(r) = {1 \over r}$. 
Inspired by an argument of Topping \cite{Topping:MCF} in the flat case and in order to simplify our fundamental identity, we make the following preliminary observation.

\begin{lemma}[Square completion for the radial density]
\label{lem:120}
On the surface $\Sigma$, one has
\be
\delta_p =
- \frac{2}{r^2} \la n, \del_r \ra^2 + \frac{2}{r} \la n, \del_r \ra \, H
\leq {1 \over 2} H^2,
\ee
with equality when $\langle n,\del_r\rangle/r=H/2$. Moreover, if the ambient manifold is the flat Euclidean space, then this inequality is saturated if $\Sigma$ is a round sphere (or a plane if the surface were not assumed to be closed).
\end{lemma}

\begin{proof} We observe that, by completing the square,
\be
\begin{aligned}
- \frac{2}{r^2} \la n, \del_r \ra^2 + \frac{2}{r} \la n, \del_r \ra \, H
&= - 2 \left| {1 \over r} \la n, \del_r \ra - {1 \over 2} H \right|^2 + {1 \over 2} H^2
 \leq {1 \over 2} H^2,
\end{aligned}
\ee
as required. A straightforward calculation shows that the term ${1 \over r} \la n, \del_r \ra - {1 \over 2} H$ vanishes\footnote{For a general surface, the directional normal curvature at $p$ need not equal its mean curvature, so equality need not hold asymptotically.} for a sphere or plane in flat $\RR^3$. 
\end{proof}

The identity~\eqref{Simon2} and Lemma~\ref{lem:120} suggest that we introduce the {\sl deviation functional} (as measured from the sphere)
\begin{subequations}
\label{nota} 
\bel{DevD}
\Dscr_p(\Sigma) \coloneqq \int_\Sigma \Big( {1 \over 2} H^2 - \delta_p \Big) \, d\Acal_\Sigma \geq 0,
\ee
as well as the {\sl cut-locus functional}\footnote{This term is non-negative, as will be proven in Lemma~\ref{lemma:51} below.}
\bel{CutlocusC}
\Cscr_p(\Sigma) \coloneqq \int_{{\Sigma \cap \Cut_p(M)}} {1 \over r} \la N, \del_r^- - \del_r^+ \ra \, d\Lcal_\Sigma
\geq 0,
\ee
and the {\sl curvature functional}
\bel{CurvR}
\Rscr_p(\Sigma) \coloneqq \int_{\Sigma \setminus \Cut_p(M)} \rho_p \, d\Acal_\Sigma.
\ee
\end{subequations} 
With this notation, we can restate Proposition~\ref{prop:fundamental} in the following form.

\begin{proposition}[Decomposition of the Willmore energy]
\label{prop:Simon4}
The Willmore energy of a closed surface immersed in a Riemannian $3$-manifold under the equation~\eqref{cutSigma} satisfies the identity
\bel{eq:2000}
{1 \over 2} \Wscr(\Sigma)
= 2 \pi \, k_p(\Sigma)  
 + \Dscr_p(\Sigma) + \Rscr_p(\Sigma) - \Cscr_p(\Sigma)
\ee
with the notation \eqref{nota}. 
\end{proposition}

Conformal inversion gives another interpretation of the square deviation; see~\cite{Kusner}. We continue with the radial divergence identity.

At this juncture, we can make the following observations: 
\begin{itemize}

\item \emph{Deviation contribution.}  By Lemma~\ref{lem:120}, the term $ \Dscr_p(\Sigma)$ has a favorable sign, as far as getting a lower bound on  the Willmore energy is concerned.

\item \emph{Curvature contribution.} By imposing appropriate curvature bounds on the ambient space $(M,g)$, we expect to control the term $\Rscr_p(\Sigma)$. It will turn out that an {\sl upper} bound on the sectional curvature of $g$ is sufficient to derive a {\sl lower} bound on this term.

\item \emph{Cut-locus contribution.} The nonnegative cut-locus functional $\Cscr_p(\Sigma)$ enters the identity with a minus sign. Our investigation is three-fold: 

\bei 

\item When the surface has a ``small'' diameter in comparison to the injectivity radius, then  $\Cscr_p(\Sigma)$ is irrelevant. 

\item Otherwise, the cut-locus contribution must be kept and estimated directly, leading to Theorem~\ref{theo:34-e}. 

\item Another strategy, leading to Theorem~\ref{thm:Will1}, is obtained by choosing the test-function $\psi$ to vanish for all $r \geq \inj(M)$, so that the term $\Cscr_p(\Sigma)$ vanishes too.

\eei

\end{itemize}
 

\subsection*{Surfaces in ambient manifolds with non-positive curvature}

Let us now restrict attention to spaces with non-positive curvature, so that we can use Hessian comparison in order to control the curvature contribution. 

\begin{lemma}[Non-positively curved space]
\label{deltapos}
If the ambient space $(M,g)$ has non-positive sectional curvature, then the following lower bound holds
\be
\rho_p \ge 0 \quad \text{ almost everywhere on }\Sigma\setminus\big(\Cut_p(M)\cup\{p\}\big)
\ee
and, consequently, $\Rscr_p \geq 0$.
\end{lemma}

\begin{proof} Our analysis takes place away from $p$ and its cut locus. In non-positively curved manifolds, the Hessian comparison theorem~\cite{Petersen}  
for the distance function $r=\dist_p$ implies that
$\Hess r \ge {1 \over r} g_r$ with $g_r \coloneqq g - dr \otimes dr$
or, equivalently,
\be
\Hess \big( r^2 /2\big) =
dr \otimes dr + r \Hess r \ge dr \otimes dr + g_r = g.
\ee
By taking the trace in this inequality (as a bilinear form), it follows that
\be
\rho_p = \frac1{r^2}\tr_{g_\Sigma}\Big(\Hess(r^2/2)-g\Big)
=
\frac1{r^2}\sum_{i=1}^2\Big(\Hess(r^2/2)-g\Big)(e_i,e_i)\geq0,
\ee
where $\{ e_1, e_2 \}$ is any orthonormal frame for the metric $g_\Sigma$.
\end{proof}

Combining Proposition~\ref{prop:Simon4} and Lemma ~\ref{deltapos}, we therefore have the following conclusion. 

\begin{proposition}[Non-positively curved spaces]
\label{theo:34}
Let $\Sigma$ be a closed surface immersed in a non-positively curved space $(M,g)$ satisfying the equation~\eqref{cutSigma}. 
If $p \in \Sigma$ is a $k$-fold point of this surface, then one has
\bel{Simon5}
{1 \over 2} \Wscr(\Sigma) \geq
 2 \pi k_p(\Sigma)  
 + \Dscr_p(\Sigma) - \Cscr_p(\Sigma), 
\ee
in which $\Dscr_p(\Sigma)$ is non-negative. In particular, if $M$ has infinite injectivity radius or, more generally, if 
the diameter of $\Sigma$ is less than $\inj_p(M)$, 
then $\Cscr_p(\Sigma)$ vanishes and, if $p$ achieves $k_p(\Sigma)=k(\Sigma)$, one concludes that 
\bel{Simon5B}
{1 \over 4 \pi} \Wscr(\Sigma) \geq k(\Sigma). 
\ee
\end{proposition}
 
In the case where $M$ is simply connected, there are no geodesic loops, and we reach the following conclusion. 

\begin{corollary}[A generalized Li--Yau inequality for Hadamard spaces]
\label{corollary:LiYau}
Let $(M, g)$ be a complete, simply-connected Riemannian $3$-manifold with non-positive curvature, and $\Sigma$ be a closed, oriented, immersed surface. Then the Willmore energy of $\Sigma$ is uniformly bounded below
\bel{eq:GLY} 
\Wscr(\Sigma) \geq 4 \pi \, k(\Sigma)
\ee
by a multiple of the largest number of self-intersections of the surface. In particular, if the Willmore energy is less than $8\pi$ then the surface admits no self-intersection and the surface is actually embedded.  
\end{corollary}

In Euclidean space, this is the Li--Yau inequality~\cite[p.~286]{LiYau:Willmore} (see also \cite{Topping:MCF}). For the isoperimetric comparison in dimension three, see~\cite{Kleiner}. For embedded surfaces, see also the conformal argument of Ritor\'e~\cite[Section~4.1]{Ritore}. The corresponding Willmore lower bound also holds for integral varifolds in three-dimensional Cartan--Hadamard manifolds; see~\cite[Lemma~6.7]{Schulze}. More recent Li--Yau inequalities for varifolds supported in a geodesically star-shaped neighbourhood are established in~\cite[Theorem~1.7]{Scharrer}. Equality in our decomposition requires $\Dscr_p=\Rscr_p=0$, a condition along the surface and its radial geodesics. 

Considering now manifolds with finite injectivity radius, we are going to show that ${\Cscr_p(\Sigma) \geq 0}$ and also derive an upper bound.  
We first recall some properties of the cut locus of a complete Riemannian manifold that we will require (see~\cite{Ozols, Hebda}). Given $p \in M$, we denote by $Q(p) \subseteq \Cut_p(M)$ the set of points $q \in \Cut_p(M)$ such that $q$ is conjugate to $p$ along at least one minimizing geodesic from $p$ to $q$. Given any $q \in \Cut_p(M) \setminus Q(p)$, there exist $k(q)$ minimizing geodesics from $p$ to $q$ of length $d(p, q)$, where $k(q)$ is the \emph{order\/} of $q$. For any $q \in \Cut_p(M) \setminus Q(p)$, $k(q)$ is finite and greater than or equal to $2$~\cite{Ozols}. A point $q \in \Cut_p(M) \setminus Q(p)$ of order $2$ is referred to as a \emph{cleave point}~\cite{Hebda}.

We denote by $K \subseteq \Cut_p(M)$ the set of cleave points. Note that
$K$ is a relatively open subset of $\Cut_p(M)$ and forms a smooth $(n-1)$-dimensional sub-manifold of $M$. The contribution of the remaining cut points is included through the surface measure $\nu_p$; ambient negligibility does not imply negligibility for length on the immersed surface. Finally, given a point $q \in K$, there exist two minimizing geodesics $\gamma_1, \gamma_2$ from $p$ to $q$, which we assume to be parametrized by arc-length. It follows~\cite[Remark~2.5]{Ozols} that the tangent vectors to $\gamma_1$ and $\gamma_2$ at the point $q$ intersect $K$ at the same angle, in the sense that $\dot{\gamma}_1(q) - \dot{\gamma}_2(q)$ is orthogonal to $T_q K$. In our notation from earlier, this implies that at the points where the set $\Sigma$ intersects $K \subset \Cut_p(M)$, the jump in the radial vector $ \del_r^+ - \del_r^-$ is {\sl normal} to the cut locus. At other surface jump points, we use the intrinsic traces in Lemma~\ref{cut-measure-lemma}, below.

These properties lead us immediately  to the following property.

\begin{lemma}[Lower and upper bounds on the cut-locus contribution]
\label{lemma:51}
One has 
\be
0 \leq \la \del_r^- - \del_r^+, N \ra \leq 2
\qquad\text{ for }\Hcal^1_{g_\Sigma}\text{-almost every point of }E_p
\ee
and, in particular, the cut-locus contribution $\Cscr_p(\Sigma)$ is non-negative.
\end{lemma}
\begin{proof}
Lemma~\ref{cut-measure-lemma}, below, gives $\langle\del_r^--\del_r^+,N\rangle=\beta_p\in[0,2]$ almost everywhere for surface length. Since $r\geq\inj_p(M)\geq\inj(M)$ on the cut set, $0\leq\Cscr_p\leq2|\Sigma\cap\Cut_p(M)|/\inj(M)$. If this length is infinite, the length estimate is trivial; exact decompositions use the finite-length assumption.
\end{proof}

By combining Proposition~\ref{theo:34} and Lemma~\ref{lemma:51}, 
we have thus reached the proof of Theorem~\ref{theo:34-e} stated in the introduction.


\section{Ambient spaces with non-positive curvature}
\label{sec:33}
 
 \subsection*{Proof of Theorem~\ref{thm:Will1}} 
 
We now derive our main lower bound for the Willmore energy, which was stated in Theorem~\ref{thm:Will1}.
From now on, we avoid the cut locus by working within a geodesic ball $B_R(p)$ and assuming that $R < \dist_p(\Sigma \cap \Cut_p(M))$. 
Given a continuous function $f:[0,1]\to[0,1]$, locally Lipschitz on $(0,1]$ and with integrable radial densities, satisfying
\bel{eq:condi-f}
f(0) = 0, \qquad f(1) = 1, \qquad f \text{ is non-decreasing},
\ee
we introduce the function
\be
\psi(r) =
\begin{cases}
{1 \over r} \Big( 1 - f\left( \frac{r}{R} \right) \Big), & r \leq R,
\\
0, & r > R,
\end{cases}
\ee
and we work with the corresponding (weighted) radial vector field $\psi(r) \del_r$, associated with radial geodesics from $p$.

\begin{proposition}[Lower bound for the Willmore energy in non-positively curved spaces. II]
Given a closed surface $\Sigma$ immersed in a non-positively curved space $(M,g)$
and given a point $p\in\varphi(\Sigma)$, one has
\bel{eq:400}
{1 \over 2} \Wscr(\Sigma) 
 \ge 2 \pi k_p - \int_{\Sigma \cap B_R(p)} \frac{f'\left( r/R \right)}{R r} \, d\Acal_\Sigma, 
\ee
provided $R < \dist_p(\Sigma \cap \Cut_p(M)\big)$ and the function $f$ satisfies the conditions \eqref{eq:condi-f}.
\end{proposition}

\begin{proof} 
By construction, we have $\psi\equiv0$ before the cut set and can apply the proof of Proposition~\ref{prop:fundamental} just on its support, without any cut-length condition outside that support and thus by substituting this function $\psi$ into the identity of Proposition~\ref{prop:fundamental}, we find
\be
\label{eq:ident}
2 \pi k_p = \int_{\Sigma \cap B_R(p)} \big( \delta_p - \rho_p \big) \, d\Acal_\Sigma,
\ee
where, in view of \eqref{eq:defchitheta},
\be
{1 \over 2} 
\delta_p = - \left( \frac{1-f(r/R)}{r^2} + \frac{f'(r/R)}{R r} \right) \la n, \del_r \ra^2 + {1 \over r} \left( 1 - f(r/R) \right) \la \del_r, H n \ra,
\ee
and 
\be
\aligned
\rho_p =
&  {1 \over r} \left( 1 - f\left( \frac{r}{R} \right) \right) 
       \Big( \Delta r - \left( \Hess r \right)(n, n) \Big)
\\
& - \left( \frac{1-f(r/R)}{r^2} + \frac{f'(r/R)}{R r} \right) \big( 1 + \la n, \del_r \ra^2 \big).
\endaligned
\ee
As already observed in Lemma~\ref{deltapos}, for non-positively curved manifolds, the Hessian comparison theorem applies, and implies that
\be
\aligned
\Delta r - \left( \Hess r \right)(n, n) 
& = \tr_{g_\Sigma} \Hess r 
   \ge {1 \over r} \tr_{g_\Sigma} g_r = {1 \over r} \left( 1 + \la n, \del_r \ra^2 \right). 
\endaligned
\ee
Hence, the curvature contribution is bounded below by 
\bel{eq:9285}
\rho_p \ge - \frac{f'(r/R)}{R r} \left( 1 + \la n, \del_r \ra^2 \right).
\ee

Next, from \eqref{eq:ident} and after re-organizing the expression, 
we find
\be
\aligned
2 \pi k_p
\leq \int_{\Sigma \cap B_R(p)} \hskip-.2cm \Bigg( \frac{f'\left( r/R \right)}{R r} 
& - \la n, \del_r \ra^2 \Big( \frac{2(1-f( r/R))}{r^2} + \frac{f'(r/R)}{R r} \Big) 
\\
&
+ \frac{2(1 - f(r/R))}{r} \la \del_r, H n \ra \Bigg) d\Acal_\Sigma
\endaligned
\ee
and, therefore, 
\bel{ndrsquare}
\aligned
2 \pi k_p
\leq - \mathcal{E}_R[\epsilon] +  \int_{\Sigma \cap B_R(p)} \Bigg( \frac{f'\left( r/R \right)}{R r} 
& + \la n, \del_r \ra^2 \left( - 2 \frac{1-f\left( r/R \right)}{r^2} - \frac{f'(r/R)}{R r} \right)
\\
& + {1 \over r} \left( 1 - f(r/R) \right) \left( {1 \over \eps} H^2 + \eps \la n, \del_r \ra^2 \right) \Bigg) d\Acal_\Sigma,
\endaligned
\ee
where, at this stage, $\eps=\eps(r) > 0$ is an arbitrary function for which the displayed integrals are finite, and we have introduced the nonnegative remainder (further studied at the end of this section) 
\bel{error}
\mathcal{E}_R[\epsilon] \coloneqq \int_{\Sigma \cap B_R(p)} {1 \over r} \left( 1 - f(r/R) \right) \left( \sqrt{\eps} \la n, \del_r \ra - {1 \over \sqrt{\eps}} H \right)^2 d\Acal_\Sigma \ge 0.
\ee
The nonnegative remainder comes from the identity $2xy=\eps x^2+y^2/\eps-(\sqrt\eps\,x-y/\sqrt\eps)^2$, with $x=\langle n,\del_r\rangle$ and $y=H$.

The function $\eps$ is now chosen in order to eliminate the term $\la n, \del_r \ra^2$ in~\eqref{ndrsquare}, and thus we take
\bel{epchoice}
\eps(r) \coloneqq \frac{f'\left( r/R \right)}{R (1-f(r/R))} + \frac{2}{r} >0.
\ee
Where $f(r/R)=1$, we put $\eps(r)=2/r$; monotonicity gives $f\prime(r/R)=0$ almost everywhere there. The zero-over-zero quotient below is defined to be zero. This choice of $\eps$ in 
\eqref{ndrsquare} leads us to the inequality
\be
2 \pi k_p \leq \int_{\Sigma \cap B_R(p)} \left( \frac{f'\left( r/R \right)}{R r}
+ \frac{\left( 1-f\left( r/R \right) \right)^2 }{2(1-f(r/R)) + \frac{r}{R} f'(r/R)} H^2 \right) d\Acal_\Sigma.
\ee
Finally, in view of our conditions \eqref{eq:condi-f} on the function $f$, we have
\be
0 \leq \frac{\left( 1-f\left( r/R \right) \right)^2 }{2(1-f(r/R)) + \frac{r}{R} f'(r/R)} \leq {1 \over 2}
\qquad\text{ for almost every }r\in(0,R).
\ee
For the power cutoffs $f(x)=x^n$, $n>0$, this coefficient tends to zero as $r\uparrow R$ and to $1/2$ as $r\downarrow0$. 
Therefore, we have obtained 
\bel{TGTBT}
2 \pi k_p \leq \int_{\Sigma \cap B_R(p)} \frac{f'\left( r/R \right)}{R r} \, d\Acal_\Sigma
+ {1 \over 2} \int_{\Sigma \cap B_R(p)} H^2 \, d\Acal_\Sigma.
\ee
For the step-cutoff limit, the same calculation also keeps the tangential gradient factor. With $\delta_p^0,\rho_p^0$ denoting the densities for $f=0$, direct substitution gives $\delta_p-\rho_p=(1-f)(\delta_p^0-\rho_p^0)+f'(r/R)|\nabla_\Sigma u|^2/(Rr)$. Thus
\bel{cutoff-gradient-retained}
2\pi k_p\leq\frac12\int_{\Sigma\cap B_R(p)}(1-f(r/R))H^2\,d\Acal_\Sigma
+\int_{\Sigma\cap B_R(p)}\frac{f'(r/R)|\nabla_\Sigma u|^2}{Rr}\,d\Acal_\Sigma.
\ee
\end{proof} 

\begin{proof}[Proof of Theorem~\ref{thm:Will1}] 
We derive first two special cases, as stated in \eqref{eq:191}, which are of particular interest. Letting $f(x) = x$ in \eqref{eq:400}, we deduce that
\bel{eq:440}
\aligned
{1 \over 2} \Wscr(\Sigma) \geq 
{1 \over 2} \int_{\Sigma \cap B_R(p)} H^2 \, d\Acal_\Sigma 
& \geq 2 \pi k_p - {1 \over R} \int_{\Sigma \cap B_R(p)} \frac{d\Acal_\Sigma}{r} 
\\ 
& \geq 2 \pi k_p - {1 \over R} \int_\Sigma \frac{d\Acal_\Sigma}{r} 
\endaligned
\ee
for all $R < \dist_p(\Sigma \cap \Cut_p(M))$. Next, taking $f(x) = x^2$, we find (again for all such $R$)
\be
\aligned
{1 \over 2} \Wscr(\Sigma) \geq {1 \over 2} \int_{\Sigma \cap B_R(p)} H^2 \, d\Acal_\Sigma 
& \ge 2 \pi k_p - \frac{2}{R^2} |\Sigma \cap B_R(p)|
\\
& \geq 2 \pi k_p - \frac{2}{R^2} |\Sigma|.
\endaligned
\ee
The second inequality is an analogue (in our approach) of the estimate established earlier by Topping in flat space~\cite[Equation~(22)]{Topping:MCF}. More generally, by taking $f(x) = x^n$ in \eqref{eq:400} for any $n > 0$, we arrive at \eqref{eq:190}.

Finally, the strongest statement\footnote{We are thankful to Bruno Le Floch for this observation.} is found by choosing $f$ as a limit of smooth nondecreasing cutoffs approaching the step-function: 
\be
f(x) = \begin{cases}
0, \qquad & x\leq r_0/R, 
\\
1, \qquad & x>r_0/R, 
\end{cases}
\ee
where $r_0$ is a regular value of $u$ with $r_0<R$ and $r_0< \dist_p(\Sigma \cap \Cut_p(M)\big)$. Coarea applied to the sharper cutoff term in the equation~\eqref{cutoff-gradient-retained}, above, gives  
\be
\aligned
{1 \over 2} \Wscr(\Sigma) 
&\geq2\pi k_p-\frac1{r_0}\int_{u=r_0}|\nabla_\Sigma u|\,d\Lcal_\Sigma\\ &\geq2\pi k_p-\frac{|\Sigma\cap S_{r_0}(p)|}{r_0} 
\endaligned
\ee
and, by optimizing over $r_0$,  this completes the proof of Theorem~\ref{thm:Will1}.
\end{proof}  


\subsection*{The quadratic remainder}

We now investigate the nonnegative remainder'' $\mathcal{E}_R[\epsilon]$, defined in~\eqref{error}. Observe that $\mathcal{E}_R[\epsilon] = 0$ if we impose that the mean curvature satisfies the identity $H = \eps \la n, \del_r \ra$. Choosing $\epsilon$ as in~\eqref{epchoice}, we deduce that the remainder vanishes if the surface satisfies
\bel{Hndr}
H = \left( \frac{f'\left( r/R \right)}{R (1-f(r/R))} + \frac{2}{r} \right) \la n, \del_r \ra.
\ee
As noted previously, if $\Sigma$ is a round sphere in $\RR^3$, then $H = \frac{2}{r} \la n, \del_r \ra$. As such, in flat space, we may take $R = + \infty$ and the remainder vanishes (for $f\equiv0$) if $\Sigma$ is a sphere.

For an arbitrary function $f$, we ask whether there exists a rotationally symmetric surface in $\RR^3$ that satisfies~\eqref{Hndr}. In particular, we seek a surface $\Sigma$ that is the graph of a function $z = A(t)$, where $t \coloneqq \sqrt{x^2 + y^2}$. A simple calculation shows that
$n=\big(A'(t)\del_t-\del_z\big)/\sqrt{1+A'(t)^2}$
and, hence,
\be
H = {1 \over 2} \dv n = {1 \over 2 t} \del_{t} \left( t \frac{A'(t)}{\sqrt{1+A'(t)^2}} \right).
\ee
Therefore, the equation \eqref{Hndr} takes the form
\bel{eq:450} 
{1 \over 2 t} \del_{t} \left( t \frac{A'(t)}{\sqrt{1+A'(t)^2}} \right) =
\left( \frac{f'\left( r/R \right)}{R (1-f(r/R))} + \frac{2}{r} \right) \frac{t A'(t) - A(t)}{r \sqrt{1+A'(t)^2}},
\ee
where $r = \sqrt{t^2 + A(t)^2}$. Given a function $f$, we could in principle solve this equation for the function $A$. In the limiting case $f \equiv 0$, we recover the solution corresponding to the round two-sphere of radius $a$ in $\RR^3$, namely $A(t) = a - \sqrt{a^2 - t^2}$. It would be interesting to determine whether, for sufficiently large $R$, the equation~\eqref{eq:450} admits solutions close to the round sphere. 


\subsection*{The example of the long cylinder}

Consider now the example where $(M,g)$ is a three-dimensional cylinder, viewed as $\RR^3$ with coordinates $x,y,z \in \RR$, identified with period $2a$ in the $x$-direction. Let $\Sigma$ be the two-sphere of radius $a$ centered at the point $(a, 0, 0)$ and let $p = (0, 0, 0)$ be  the point of two-fold self-intersection of $\Sigma$. Under these conditions, the injectivity radius (at $p$) is $\inj_p(M) = a$, and $\Sigma$ has mean curvature $H = 1/a$. So we have
\be
{1 \over 2} \int_{\Sigma \cap B_a(p)} H^2 \, d\Acal_\Sigma
= 2 \pi \int_0^{\pi/3} \sin\theta \, d\theta
= \pi
\ee
and we can also calculate
\be
\begin{aligned}
{1 \over \inj_p(M)} \int_{\Sigma \cap B_a(p)} {1 \over \dist_p(\cdot)} \, d\Acal_\Sigma
&=
4 \pi \int_0^{\pi/3} \frac{\sin\theta}{\sqrt{2} \sqrt{1 - \cos\theta}} \, d\theta
= 4 \pi.
\end{aligned}
\ee
Since $k_p(\Sigma) = 2$, our calculations confirm that the inequality \eqref{eq:440} is satisfied for this example.


The bounds are nontrivial for surfaces crossing the cut locus. In the same flat cylinder of period $2a$, take instead a sphere of radius $b$ centered at $(b,0,0)$, with $a/2<b<5a/8$ and $p=(0,0,0)$. Its projection is embedded and crosses the cut plane $x=a$ (see Figure~\ref{fig:cylinder-cut-locus}). That intersection has length $2\pi\sqrt{2ab-a^2}$ and $\inj(M)=a$. Theorem~\ref{theo:34-e} gives the strictly positive bound
\be
\Wscr(\Sigma)\geq4\pi\left(1-2\sqrt{2b/a-1}\right)>0.
\ee

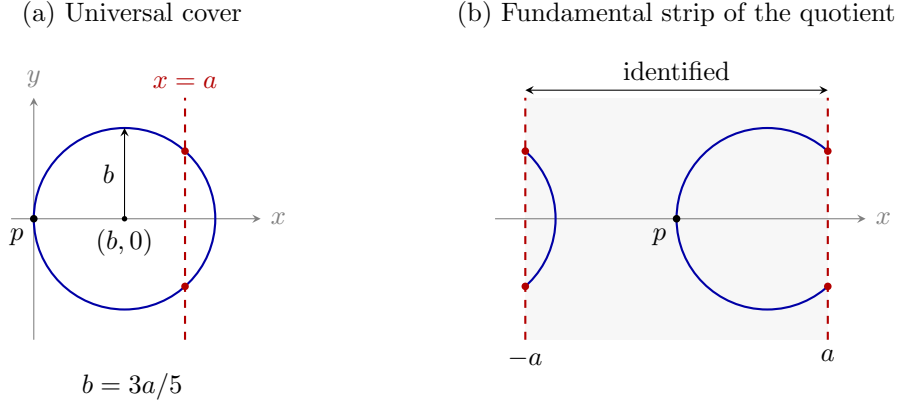
\begin{figure}[htbp]
\centering
\begin{tikzpicture}[x=2cm,y=2cm,>=stealth,font=\small]
\begin{scope}
\node at (.65,1.35) {(a) Universal cover};
\draw[->,gray] (-.15,0)--(1.5,0) node[right] {$x$};
\draw[->,gray] (0,-.8)--(0,.8) node[above] {$y$};
\draw[red!70!black,dashed,thick] (1,-.8)--(1,.8);
\node[red!70!black,above] at (1,.8) {$x=a$};
\draw[blue!65!black,thick] (.6,0) circle (.6);
\fill (0,0) circle (.025) node[below left] {$p$};
\fill (.6,0) circle (.018) node[below] {$(b,0)$};
\draw[->] (.6,0)--(.6,.6) node[midway,left] {$b$};
\fill[red!70!black] (1,.447214) circle (.025);
\fill[red!70!black] (1,-.447214) circle (.025);
\node[below] at (.65,-.95) {$b=3a/5$};
\end{scope}
\begin{scope}[xshift=8.5cm]
\node at (0,1.35) {(b) Fundamental strip of the quotient};
\fill[gray!6] (-1,-.8) rectangle (1,.8);
\draw[red!70!black,dashed,thick] (-1,-.8)--(-1,.8);
\draw[red!70!black,dashed,thick] (1,-.8)--(1,.8);
\draw[->,gray] (-1.2,0)--(1.25,0) node[right] {$x$};
\begin{scope}
\clip (-1,-.8) rectangle (1,.8);
\draw[blue!65!black,thick] (.6,0) circle (.6);
\draw[blue!65!black,thick] (-1.4,0) circle (.6);
\end{scope}
\fill (0,0) circle (.025) node[below left] {$p$};
\foreach \x in {-1,1}{
\fill[red!70!black] (\x,.447214) circle (.025);
\fill[red!70!black] (\x,-.447214) circle (.025);}
\draw[<->] (-1,.85)--(1,.85) node[midway,above] {identified};
\node[below] at (-1,-.8) {$-a$};
\node[below] at (1,-.8) {$a$};
\end{scope}
\end{tikzpicture}
\caption{An embedded sphere crossing the cut locus in the flat cylinder $M=\mathbb R^3/(2a\mathbb Z\times\{0\}\times\{0\})$. Both panels show the section $z=0$ for the sphere of radius $b=3a/5$ centered at $(b,0,0)$. In (a), the dashed line is a lift of $\Cut_p(M)$. In (b), the edges $x=-a$ and $x=a$ are identified: the cap beyond $x=a$ reappears at the left edge. The red points represent the cut-intersection circle, of length $2\pi\sqrt{2ab-a^2}$. Thus the sphere is embedded while meeting the cut locus, and the lower bound above, including the cut-locus contribution, is strictly positive.}
\label{fig:cylinder-cut-locus}
\end{figure}

\subsection*{Bounds on the injectivity radius}

Given our results so far, the following statement is immediate and implies, for example, that if a surface has many self-intersections and its mean curvature is small (in the $L^2$ sense), then the ambient manifold must have rather short geodesic loops. 

\begin{corollary}[Bounds on the injectivity radius]
 Let $(M, g)$ be a complete Riemannian $3$-manifold with non-positive sectional curvature.
Suppose that there exists a closed, oriented, immersed surface $\Sigma$ whose largest self-intersection number $k(\Sigma)$ is sufficiently large or Willmore energy is sufficiently small so that $2 \pi k(\Sigma) - {1 \over 2} \int_\Sigma H^2 \, d\Acal_\Sigma > 0$.
Then, the injectivity radius $\inj(M)$ along the immersion is bounded above:
\be
\inj(M) \leq \frac{\max_{p \in \Sigma} \int_\Sigma {1 \over \dist_p} \, d\Acal_\Sigma}{2 \pi k(\Sigma) - {1 \over 2} \int_\Sigma H^2 \, d\Acal_\Sigma}.
\ee
\end{corollary}


\section{Ambient spaces with curvature bounded above}
\label{sec:4}

\subsection*{A second decomposition of the Willmore energy}

In Sections~\ref{sec:4} and \ref{sec:WF}, we put $\tau_K=\pi/(2\sqrt K)$ for $K>0$ and $\tau_K=+\infty$ for $K\leq0$. Global weighted integrals require $\max_\Sigma\dist_p<\tau_K$, and exact cut decompositions require the equation~\eqref{cutSigma}. The local estimates only use these conditions on their region of integration. 

The considerations of the previous section were aimed at studying metrics with non-positive curvature, and we may now generalize these results to a Riemannian manifold $(M,g)$ whose sectional curvature $\Sec_{(M,g)}$ is bounded above, that is,  
\bel{eq:500}
\Sec_{(M,g)} \leq K
\ee
for some real $K$. As before, let $\Sigma$ be an immersed surface. 
We denote the induced metric $g_\Sigma$ and we fix a point $p \in \Sigma$ achieving the largest number of self-intersections. We consider the distance function $r=\dist_p(\cdot)$ and introduce the radial vector field (defined away from $p$ and its cut locus) 
\bel{eq:510} 
X_K \coloneqq {1 \over \sn_K(r)} \, \del_r, 
\ee
which involves the weight function ($K \in \RR$) 
\be
\sn_K(r) \coloneqq
\begin{cases}
\frac{\sin(\sqrt K\,r)}{\sqrt K}, &K>0,
\\
r, &K = 0,
\\
\frac{\sinh(\sqrt{|K|}\,r)}{\sqrt{|K|}}, &K<0.
\end{cases}
\ee

This normalization gives $\sn_K\prime(0)=1$ and $r/\sn_K(r)\to1$. On the stated radius range, $\sn_K$ and $\sn_K\prime$ are positive. In order to proceed, we recall that by setting 
\be
\md_K(r) \coloneqq \int_0^r \, \sn_K(t) dt =
\begin{cases}
{1 \over K} \left( 1 - \cos \sqrt{K} r \right), &K > 0,
\\
{1 \over 2} r^2, &K = 0,
\\
{1 \over |K|} \left( \cosh \sqrt{|K|} r - 1 \right), &K < 0, 
\end{cases}
\ee
and relying on the upper curvature bound \eqref{eq:500}, a standard Hessian comparison theorem (see, for instance,~\cite{Petersen})
yields 
\be
\Hess \md_K(r) \ge \sn_K^{\prime}(r) \, g
\ee
and, for the trace, this implies that 
\bel{eq:444}
\tr_{g_\Sigma} \left( \Hess \md_K(r) \right) \ge \sn_K^{\prime}(r) \left( \tr_{g_\Sigma} g \right) = 2 \, \sn_K^{\prime}(r).
\ee
This motivates us to introduce the vector field \eqref{eq:510}. 

Let us introduce the {\sl deviation functional} 
\begin{subequations}
\label{nota-K} 
\bel{DevD-K}
\aligned
\Dscr_{K,p}(\Sigma) 
& \coloneqq \int_\Sigma \Big( {1 \over 2} H^2 - \delta_{K,p} \Big) \, d\Acal_\Sigma\text{ (with no prescribed sign)},
\\
\delta_{K,p}
&\coloneqq
- 2 \frac{\sn_K^{\prime}(r)}{\sn_K(r)^2} \langle n, \del_r \rangle^2
+ \frac{2}{\sn_K(r)} \la \del_r, H n \ra, 
\endaligned
\ee
as well as the {\sl cut-locus functional}\footnote{This term is non-negative by Lemma~\ref{lemma:51}.}
\bel{CutlocusC-K}
\Cscr_{K,p}(\Sigma) \coloneqq \int_{{\Sigma \cap \Cut_p(M)}} \frac1{\sn_K(r)} \la N, \del_r^- - \del_r^+ \ra \, d\Lcal_\Sigma
\geq 0,
\ee
and the {\sl curvature functional}
\bel{CurvR-K}
\aligned
\Rscr_{K,p}(\Sigma) 
&\coloneqq \int_{\Sigma \setminus \Cut_p(M)} \rho_{K,p} \, d\Acal_\Sigma,
\\ 
\rho_{K,p} 
&\coloneqq {1 \over \sn_K(r)^2} \tr_{g_\Sigma} \Big( \Hess \md_K(r) - \sn_K^{\prime}(r) \, g \Big).
\endaligned
\ee
\end{subequations} 
With this notation, we obtain the following identity, which we next apply under the assumption \eqref{eq:500}. 

\begin{proposition}[Decomposition of the Willmore energy]
\label{prop:Simon4-K}
Under the radius and cut-length assumptions above, for any real $K$, the Willmore energy of a closed surface immersed in a Riemannian $3$-manifold satisfies the identity
\bel{eq:515}
{1 \over 2} \Wscr(\Sigma)
= 2 \pi \, k_p(\Sigma) 
 + \Dscr_{K,p}(\Sigma) + \Rscr_{K,p}(\Sigma) - \Cscr_{K,p}(\Sigma)
\ee
with the notation above. 
\end{proposition}

\begin{proof}
Similar calculations to those in Section \ref{sec:3} show that
\be
\dive_\Sigma \left( \sn_K(r) \del_r \right) = \tr_{g_\Sigma} \left( \Hess \md_K(r) \right),
\ee
and, therefore, 
\be
\dive_\Sigma X_K = 
\dive_\Sigma\left({1\over\sn_K(r)}\del_r\right) 
 = 2 \, \frac{\sn_K^{\prime}(r)}{\sn_K(r)^2} \langle n, \del_r \rangle^2 + \rho_{K,p}.
\ee
Applying Proposition~\ref{prop:fundamental}, we deduce that
\bel{eq:515-deux}
\aligned
2 \pi k  
= 
&
\int_\Sigma \delta_{K,p} \, d\Acal_\Sigma - \int_{\Sigma \setminus \Cut_p(M)} \rho_{K,p} \, d\Acal_\Sigma
\\
&- \int_{{\Sigma \cap \Cut_p(M)}} {1 \over \sn_K(r)} \la \del_r^+ - \del_r^-, N \ra \, d\Lcal_\Sigma. 
\endaligned
\ee
\end{proof}

A similar calculation to that in the proof of Lemma~\ref{deltapos} gives the following result.

\begin{lemma}[Sign of the curvature contribution]
\label{rem:delta}
Under the curvature assumption \eqref{eq:500}, one has $\rho_{K,p} \geq 0$ and, consequently, 
$\Rscr_{K,p}(\Sigma) \geq 0$. 
Equality holds if the ambient space $(M,g)$ has constant sectional curvature $K$.
\end{lemma}

In view of the identity~\eqref{eq:515-deux}, it follows that for a distance sphere in $S^3$ we have $\int_\Sigma \delta_{K,p} \, d\Acal_\Sigma = 2 \pi$, as can be checked directly from the formulas in Example~\ref{example:S3}. Now, if we attempt to proceed as in the case of non-positive curvature, we have the following.

\begin{lemma}[Pointwise weighted bound]
\label{lemma43}
Suppose that $\Sec_{(M,g)} \leq K$ holds and in the case $K > 0$, one also requires that $0 < r < \tfrac{\pi}{2\sqrt{K}}$. Then, on the surface $\Sigma$, one has 
\bel{Hsquared}
{1 \over 2} 
\delta_{K,p}
=
- \frac{\sn_K^{\prime}(r)}{\sn_K(r)^2} \langle n, \del_r \rangle^2
+
{1 \over \sn_K(r)} \la \del_r, n \ra \, H
\leq {1 \over 4 \, \sn_K^{\prime}(r)} H^2.
\ee
\end{lemma}

\begin{proof}
We have
\be
\begin{aligned}
& - \frac{\sn_K^{\prime}(r)}{\sn_K(r)^2} \langle n, \del_r \rangle^2
+
{1 \over \sn_K(r)} \la \del_r, n \ra \, H
\\
&=
- \frac{\sn_K^{\prime}(r)}{\sn_K(r)^2}
\left| \langle n, \del_r \rangle n - {1 \over 2} \frac{\sn_K(r)}{\sn_K^{\prime}(r)} H n \right|^2
+ {1 \over 4 \, \sn_K^{\prime}(r)} H^2
\le
{1 \over 4 \, \sn_K^{\prime}(r)} H^2,
\end{aligned}
\ee
as required.
\end{proof}

We have derived an identity with $\delta_{K,p}$. However, contrary to the situation in the flat case, the term $\langle n, \del_r \rangle n - {1 \over 2} \frac{\sn_K(r)}{\sn_K^{\prime}(r)} H n$ does not vanish for distance spheres in the model spaces. As such, we cannot recover an optimal inequality using this lemma. We therefore proceed in a different fashion and introduce next a ``weighted version'' of the Willmore functional.

\begin{example}
\label{example:S3}
Let $M = S^3(\kappa^2)$ be the three-sphere of constant curvature $\kappa^2 > 0$ and $\Sigma \cong S^2 \subset M$ a distance sphere of radius $a$ with $0<a<\pi/(4\kappa)$, then
\be
\la n, \del_r \ra = \cot \kappa a \left( \frac{1-\cos \kappa r}{\sin \kappa r} \right),
\qquad 
H_{\kappa} = \kappa \cot \kappa a,
\qquad 
\delta_{\kappa} = \frac{2 \kappa^2 \cot^2 \kappa a}{\left( 1 + \cos \kappa r \right)^2}.
\ee
Similarly, for a sphere in $M = H^3(-\kappa^2)$, the hyperbolic space of curvature $K = - \kappa^2$, then for a distance sphere of radius $a$ we have
\be
\begin{aligned}
\la n, \del_r \ra = \coth \kappa a \left( \frac{\cosh \kappa r - 1}{\sinh \kappa r} \right),
\qquad 
H_{\kappa} = \kappa \coth \kappa a,
\\
\qquad 
\delta_{\kappa} = \frac{2 \kappa^2 \coth^2 \kappa a}{\left( 1 + \cosh \kappa r \right)^2}.
\end{aligned}
\ee
These formulas are derived as follows. Let $m \in S_K^3$ in the $3$-sphere be a fixed point and $\Sigma = \{ q \in S_K^3; d(m, q) = a \}$. Let $p \in \Sigma$ be fixed. Given $q \in \Sigma$, the inner product $\la n, \del_r \ra$ at $q$ may be viewed as the inner product of the unit tangent vectors to the minimizing geodesics in $S_K^3$ from $p$ to $q$ and from $m$ to $q$. This is simply the cosine of the angle $\alpha$ at $q$ of the triangle $pqm$ in $S_K^3$ with side-lengths $|pq| = r$, $|mp| = |mq| = a$. Using the cosine formula (see~\cite[Equation~(43)]{Meyer}), we have
\be
\cos \kappa a = \cos \kappa a \cos \kappa r + \sin \kappa a \sin \kappa r \cos\alpha.
\ee
Re-arranging gives our first equality. The expression for the mean curvature follows directly from the explicit form of the metric on $S_K^3$. The formula for $\delta_K$ follows from the first two formulas by a straightforward calculation. The proof in the hyperbolic case is similar.
\end{example}


\section{A weighted Willmore functional} 
\label{sec:WF}

As observed in Lemma~\ref{rem:delta}, the curvature term $\Rscr_{K,p}(\Sigma)$ has a favorable sign, while the deviation $\Dscr_{K,p}(\Sigma)$ has no prescribed sign. Our inequality below involves a weighted version of the Willmore energy and still contains the cut-locus contribution $\Cscr_{K,p}(\Sigma)$ (which will be dealt with in the following subsection with a further modification of the vector field $X_K$). 
We thus arrive at the following conclusion.  

\begin{proposition}[Weighted Willmore energy in ambient spaces with curvature bounded above]
\label{theo:44}
Let $K\in\RR$ be a constant, assume the equation~\eqref{cutSigma}, and suppose that the ambient space $(M,g)$ satisfies the upper curvature bound \eqref{eq:500}, while $\Sigma$ is a closed, oriented, smoothly immersed surface  
such that $\diam(\Sigma) < \pi/(2 \sqrt{K})$ (if $K$ is positive). Given any $p\in\Sigma$ and $r=r_p=\dist_p(\cdot)$, set, where $H\ne0$,  
\be
\omega_{K,p} \coloneqq 2 \frac{\sn_K^{\prime}(r)}{\sn_K(r)} \, \frac{\la n,\del_r\ra}{H}.
\ee
At zeros of $H$, interpret the density in the following integral as $2\delta_{K,p}$, which involves no division by $H$. The {\rm weighted Willmore energy}
\be
{1 \over 2} \Wscrtt_{K,p}(\Sigma) \coloneqq {1 \over 2} \int_\Sigma
{( 1 - ( 1 - \omega_{K,p})^2) \over \sn_K^{\prime}(r_p)} H^2 \, d\Acal_\Sigma
\ee
satisfies the bound
\be
{1 \over 2} \Wscrtt_{K,p}(\Sigma) 
\geq 2 \pi k_p(\Sigma)  
- \Cscr_{K,p}(\Sigma).
\ee
If the ambient manifold is a constant curvature space and $\Sigma$ is a distance sphere, then this inequality becomes an {\rm equality.} 
\end{proposition}

\begin{corollary}[A weighted embedding criterion] Under the assumptions in Proposition~\ref{theo:44}, choose $p$ with $k_p(\Sigma)=k(\Sigma)$. If the diameter of $\Sigma$ is less than $\inj_p(M)$, then $\Cscr_{K,p}(\Sigma)$ vanishes and one concludes that 
\bel{Simon5B-K}
{1 \over 4 \pi} \Wscrtt_{K,p}(\Sigma) \geq k(\Sigma).   
\ee
In particular, if the weighted Willmore energy of a closed surface is less than $8\pi$, then the surface admits no self-intersection and the surface is actually embedded. 
\end{corollary} 

Note that the functional $\Wscrtt_{K,p}(\Sigma)$ need not be non-negative. Here $\omega_K$ abbreviates $\omega_{K,p}$. By letting 
$
\overline{\omega}_K\coloneqq\frac{2\sn_K\prime(r)}{1+\sn_K\prime(r)}, 
$
it follows that 
\be
\frac{(1 - (1 - \omega_K)^2) }{\sn_K\prime(r)}
= \frac{4}{(1+\sn_K\prime(r))^2} \frac{\omega_K \left( 2 - \omega_K \right)}{\overline{\omega}_K \left( 2 - \overline{\omega}_K \right)}.
\ee
By observing that $x (2-x) \leq 1$ for all real $x$, we immediately recover the earlier inequality~\eqref{Hsquared}. Furthermore, by introducing the positive and negative parts 
\be
{1 \over 2} \Wscrtt^{\pm}_{K,p}(\Sigma) 
\coloneqq {1 \over 2} \int_\Sigma
\frac{4}{(1+\sn_K\prime(r))^2} \frac{\left( \omega_K \left( 2 - \omega_K \right) \right)_{\pm}}{\overline{\omega}_K \left( 2 - \overline{\omega}_K \right)} H^2 \, d\Acal_\Sigma,
\ee
we arrive at the lower bound 
\be
{1 \over 2} \Wscrtt^+_{K,p}(\Sigma) 
\ge {1 \over 2} \Wscrtt^-_{K,p}(\Sigma)
+ 2 \pi k_p(\Sigma)  
- \Cscr_{K,p}(\Sigma).
\ee
The positive and negative densities are defined everywhere by $2(\delta_{K,p})_\pm$, including where $H=0$.
Now, the arguments developed for non-positively curved manifolds apply and we reach one of our main conclusions. 

\begin{theorem}[Lower bound under an upper curvature bound]
\label{theo:34-ff}
Let $\Sigma \hookrightarrow M$ be a closed surface  
immersed in a manifold $(M,g)$ with curvature 
$\Sec_{(M,g)} \leq K$ bounded above by a constant $K$. 
Assume $\max_\Sigma\dist_p<\tau_K$ and the equation~\eqref{cutSigma}. If $p\in\Sigma$ is a $k$-fold point, one has
\bel{Simon5-ff}
{1 \over 2} \Wscrtt_{K,p}^+(\Sigma) \geq
 2 \pi k_p(\Sigma)  
  - 2 \, \int_{E_p}\frac{d\Hcal^1_{g_\Sigma}}{\sn_K(r)}. 
\ee  
\end{theorem}

If the cut set is nonempty, monotonicity of $\sn_K$ up to $\tau_K$ gives
\be
\int_{E_p}\frac{d\Hcal^1_{g_\Sigma}}{\sn_K(r)}\leq
\frac{|\Sigma\cap\Cut_p(M)|}{\sn_K(\inj_p(M))}.
\ee
For $K\leq0$, the weaker denominator $\inj(M)$ remains valid. For $K>0$, we put $D=\max_\Sigma\dist_p<\tau_K$. Since $\Wscrtt^+_{K,p}\leq\Wscrt_{K,p}\leq\Wscr/\cos(\sqrt K D)$, the ordinary energy satisfies
\be
\Wscr(\Sigma)\geq\cos(\sqrt K D)\left(4\pi k_p(\Sigma)
-4\int_{E_p}\frac{d\Hcal^1_{g_\Sigma}}{\sn_K(r)}\right).
\ee
This gives a positive lower bound when the weighted cut length is less than $\pi k_p(\Sigma)$. If $p$ realizes $k(\Sigma)$, it also gives an embedding criterion in terms of the ordinary energy:
\be
\frac{\Wscr(\Sigma)}{4\pi\cos(\sqrt K D)}
+\frac1\pi\int_{E_p}\frac{d\Hcal^1_{g_\Sigma}}{\sn_K(r)}<2.
\ee

\begin{example}

We can carry out some of our calculations explicitly when $\Sigma$ is a distance sphere of radius $a$ in a three-sphere of constant curvature $\kappa^2$ with $0<a<\pi/(4\kappa)$. In this case, we have 
\be
\frac{\omega_K \left( 2 - \omega_K \right)}{\overline{\omega}_K \left( 2 - \overline{\omega}_K \right)} = 1
\ee
and, therefore,
\be
{1 \over 2} \Wscrtt_{K,p}(\Sigma) = {1 \over 2} \int_\Sigma
\frac{4}{(1+\cos \kappa r)^2} H^2 \, d\Acal_\Sigma.
\ee
Using again the cosine formula from~\cite[Eq.~(43)]{Meyer}, we have 
\be
\cos \kappa r = \cos^2 \kappa a + \sin^2 \kappa a \cos\theta.
\ee
From the formulas in Example~\ref{example:S3}, we then have
\be
\begin{aligned}
{1 \over 2} \Wscrtt_{K,p}(\Sigma) &= 4 \pi \cos^2 \kappa a \int_0^{\pi} \frac{\sin\theta}{(1+\cos^2 \kappa a + \sin^2 \kappa a \cos\theta)^2} \, d\theta
 = 2\pi.
\end{aligned}
\ee
In contrast, the Willmore energy of such a distance sphere takes the form
\be
\int_\Sigma H^2 \, d\Acal_\Sigma = 4 \pi \cos^2(\kappa a).
\ee
The conformally invariant energy on the same model is
\be
\int_\Sigma(H^2+\kappa^2)\,d\Acal_\Sigma=4\pi.
\ee
Thus the weighted functional has the same value on these spheres as the usual conformal energy.
\end{example}


\subsection*{A local bound under an upper curvature bound}
\label{sec:42}

We now repeat the calculation leading to the inequality~\eqref{TGTBT} but with an arbitrary upper curvature bound.  We are thus given $K \in \RR$ and a metric $g$ with sectional curvature bounded above by $K$. We choose the family of functions 
\be
\psi_K(r) \coloneqq \begin{cases} 
{1 \over \sn_K(r)} \Big( 1 - f\Big({r \over R}\Big) \Big), \quad  &0 < r < R,
\\ 
0 
&r \ge R,
\end{cases}
\ee
where $R < \dist_p( \Sigma \cap \Cut_p(M))$ is given and the function satisfies the conditions \eqref{eq:condi-f} as before.

\begin{theorem}[Weighted Willmore energy in spaces with curvature bounded above]
\label{theo:46}
Let $(M, g)$ be a complete Riemannian $3$-manifold with sectional curvature bounded above by $K>0$ and $\Sigma \hookrightarrow M$ a closed, oriented immersed surface. Fix $p$ achieving $k_p(\Sigma)=k(\Sigma)$ and require $\max_\Sigma\dist_p<\tau_K$. Then the weighted Willmore energy of $\Sigma$ 
\be
\Wscrt_{K,p} (\Sigma) = \int_\Sigma {H^2 \over \sn_K'(r)}\, d\Acal_\Sigma 
\ee
satisfies the lower bound below, where $d\coloneqq\min\{\dist_p(\Sigma\cap\Cut_p(M)),\tau_K\}$:
\bel{eq:190b}
{1 \over 4 \pi} \Wscrt_{K,p}(\Sigma) \geq k(\Sigma) 
 - \inf_{\substack{0<r<d\\r\ {\rm regular}}} {|\Sigma \cap S_{r}(p)| \over 2\pi \sn_K(r)} 
\ee
\end{theorem}

\begin{proof} In this proof, $\delta_{K,p}$ and $\rho_{K,p}$ denote the densities associated with the field $\psi_K(r)\del_r$. We have
$
2 \pi k_p = \int_{\Sigma \cap B_R(p)} \big( \delta_{K,p} - \rho_{K,p} \big) \, d\Acal_\Sigma. 
$
On one hand, we have 
\be
\aligned
\rho_{K,p}
&= {1 \over \sn_K(r)} \Big( 1 - f\Big({r \over R}\Big) \Big)
\Big( \Delta r - \left( \Hess r \right)(n, n) \Big) 
\\
& \quad 
- \Bigg( \Big(1-f\Big({r\over R}\Big)\Big)\frac{\sn_K\prime(r)}{\sn_K(r)^2} 
+ {1 \over R \, \sn_K(r)} f'\Big({r \over R} \Big)
\Bigg)  
\left( 1 + \la n, \del_r \ra^2 \right)
\endaligned
\ee
and, thanks to \eqref{eq:444}, 
\be
\aligned
\rho_{K,p}
&\geq
{1 \over \sn_K(r)} \Big( 1 - f\Big({r \over R}\Big) \Big)
 \frac{\sn_K^{\prime}(r)}{\sn_K(r)} \left( 1 + \la n, \del_r \ra^2 \right)
\\
& \quad 
- \Bigg( \Big(1-f\Big({r\over R}\Big)\Big)\frac{\sn_K\prime(r)}{\sn_K(r)^2} 
+ {1 \over R \, \sn_K(r)} f'\Big({r \over R} \Big)
\Bigg)  
\left( 1 + \la n, \del_r \ra^2 \right), 
\endaligned
\ee
which provides a lower bound for the curvature contribution: 
\be
\rho_{K,p}
\geq 
 - {1\over R \, \sn_K(r)} f'\Big({r \over R}\Big) 
\, 
\left( 1 + \la n, \del_r \ra^2 \right). 
\ee
Observe that this is consistent with our earlier estimate in Lemma~\ref{rem:delta} (in which $f$ was taken to vanish identically) as well as the equation \eqref{eq:9285} (in which $K$ was assumed to vanish). 

On the other hand, we obtain 
\be
\aligned
{1 \over 2} 
\delta_{K,p} = \, 
&
{1 \over \sn_K(r)} \left( 1 - f(r/R) \right) \la \del_r, n \ra H
\\
& - \left( (1-f(r/R)) {\sn_K'(r) \over \sn_K(r)^2}  + {f'(r/R) \over R \sn_K(r)} \right) \la n, \del_r \ra^2
\endaligned
\ee 
We now recall the integrand arising in our weighted Willmore energy $\Wscrt(\Sigma)$ 
and we claim that 
\bel{eq:keyinequa}
\delta_{K,p} + {1 \over R \, \sn_K(r)} f'\Big({r \over R}\Big) \, \la n, \del_r \ra^2 
\leq 
{1 \over 2}  
{H^2 \over \sn_K^{\prime}(r)}
\ee
The auxiliary ratio, defined only where $H\ne0$, is $\omega_{K,p} = \frac{2\sn_K^{\prime}(r)}{\sn_K(r)} \, \frac{\la n, \del_r \ra}{H}$. 
Before we proceed further, let us stress that this inequality reduces to the one derived in Lemma~\ref{lemma43} if $f \equiv 0$, while it reduces to the one in Lemma~\ref{lem:120} when $K$ is taken to be zero. 

We find 
\be
\aligned
E \coloneqq\, &
{1 \over 2} {1 \over \sn_K^{\prime}(r)} H^2
- \delta_{K,p} - {1 \over R \, \sn_K(r)} f'\Big({r \over R}\Big) \, \la n, \del_r \ra^2 
\\
= \, &
{1 \over 2} {1 \over \sn_K^{\prime}(r)} H^2
-{2 \over \sn_K(r)} \left( 1 - f(r/R) \right) \la \del_r, n \ra H 
\\
& + \left( 2 (1-f(r/R)) {\sn_K'(r) \over \sn_K(r)^2}  + {f'(r/R) \over R \sn_K(r)} \right) \la n, \del_r \ra^2, 
\endaligned
\ee
so that 
\be
\aligned
E 
\geq \, & \Bigg( 
{1 \over 2} {1\over \sn_K^{\prime}(r)} 
- {1 \over \eps} {(1 - f(r/R)) \over \sn_K(r)} \Bigg) H^2 
\\
& + \Bigg( 
- \eps {(1 - f(r/R)) \over \sn_K(r)}  + 2 (1-f(r/R)) {\sn_K'(r) \over \sn_K(r)^2} 
  + {f'(r/R) \over R \sn_K(r)} 
\Bigg) 
\la n, \del_r \ra^2.
\endaligned
\ee
On the set $f(r/R)=1$, we put $\eps(r)=2\sn_K\prime(r)/\sn_K(r)$ and set all zero-over-zero quotients below to zero. Elsewhere, we choose $\eps=\eps(r)$ so that the coefficient of $\la n, \del_r \ra^2$ vanishes, that is,  
\be
\eps(r) \coloneqq  2  {\sn_K'(r) \over \sn_K(r)}  + {f'(r/R) \over R (1-f(r/R))}.  
\ee
With this choice, we thus obtain 
\be
\aligned
E 
\geq \, & \Bigg( 
{1 \over 2} {1 \over \sn_K^{\prime}(r)} 
- {1 \over \eps(r)} {(1 - f(r/R)) \over \sn_K(r)} \Bigg) H^2,  
\endaligned
\ee 
which, since $f' \geq 0$ and $0 \leq f \leq 1$, leads us to 
\bel{TGTBT-g}
2 \pi k_p \leq \int_{\Sigma \cap B_R(p)} {1\over R \, \sn_K(r)} f'\left( r/R \right) \, d\Acal_\Sigma
+ {1 \over 2} \int_{\Sigma\cap B_R(p)}\frac{H^2}{\sn_K\prime(r)}\,d\Acal_\Sigma.
\ee
For the boundary estimate, we use the sharper identity. We write $\delta_{K,p}^f,\rho_{K,p}^f$ for the densities with the cutoff and $\delta_{K,p}^0,\rho_{K,p}^0$ for those with $f=0$. Direct substitution gives
\be
\delta_{K,p}^f-\rho_{K,p}^f=(1-f)(\delta_{K,p}^0-\rho_{K,p}^0)
+\frac{f'(r/R)}{R\sn_K(r)}|\nabla_\Sigma u|^2.
\ee
The square in Lemma~\ref{lemma43} bounds the first difference by $(1-f)H^2/(2\sn_K\prime(r))$. Smooth cutoffs approaching a step at a regular radius $t$ then give, by coarea,
\be
2\pi k_p\leq\frac12\int_{\Sigma\cap B_t(p)}\frac{H^2}{\sn_K\prime(r)}\,d\Acal_\Sigma
+\frac1{\sn_K(t)}\int_{u=t}|\nabla_\Sigma u|\,d\Lcal_\Sigma.
\ee
We use $|\nabla_\Sigma u|\leq1$ and optimize over regular $t<d$. The same local inequality remains valid without a global radius restriction if $t<\min\{\dist_p(\Sigma\cap\Cut_p(M)),\tau_K\}$; only the passage to the global weighted energy requires positivity on all of $\Sigma$. 
\end{proof}



\paragraph{Acknowledgments}
The author thanks James Grant, with whom this project was initiated in 2015. The author was partially supported by the ANR project ANR-23-CE40-0010-02 and the European Union's MSCA project EinsteinWaves (grant agreement No.~101131233).  Note on AI use: 
ChatGPT~6 was used to assist with language editing, bibliographic verification, and figure preparation.

\appendix
\section{Traces on the immersed surface}
We use the notation $u=\dist_p\circ\varphi$, $Z_p=\varphi^{-1}(p)$, and $E_p=\varphi^{-1}(\Cut_p(M))$ introduced in Section~\ref{sec:2}. The following lemma gives the traces and sign needed in the main cut identities.

\begin{lemma}[The cut measure on the parameter surface]
\label{cut-measure-lemma}
Under the equation~\eqref{cutSigma}, the distributional Laplacian of $u$ on $\Sigma\setminus Z_p$ has the form
\bel{surface-cut-measure}
\Delta_\Sigma u=(\tr_{g_\Sigma}\Hess r-2H\langle n,\del_r\rangle)\,d\Acal_\Sigma-\nu_p,
\qquad \nu_p=\beta_p\,\Hcal^1_{g_\Sigma}\!\restriction E_p,\quad 0\leq\beta_p\leq2.
\ee
The density $\beta_p$ is supported on the countably rectifiable jump set $J_p\subset E_p$ of $\nabla_\Sigma u$. At almost every point of $J_p$, let $N$ be the unit conormal in $T\Sigma$ pointing from the minus side to the plus side. The two tangential traces satisfy
\bel{surface-cut-traces}
\nabla_\Sigma u^+-\nabla_\Sigma u^-=-\beta_p N.
\ee
At an ordinary transverse crossing of a smooth ambient cut hypersurface, this is the tangential projection of the two classical radial limits. At other points the formula defines the traces intrinsically on $\Sigma$.
\end{lemma}
\begin{proof}
The ambient distance is locally Lipschitz and locally semiconcave away from $p$~\cite[Proposition~2.1 and Section~2]{MMU}. Smooth composition preserves local semiconcavity: the second-order error of the immersion is quadratic and the distance is Lipschitz. Thus $u$ is locally semiconcave on $\Sigma\setminus Z_p$, its gradient has locally bounded variation, and its singular Hessian is negative semidefinite. We use the absolutely continuous, jump, and Cantor decomposition and the trace formulas in~\cite[Chapter~3]{AFP}.

Off $E_p\cup Z_p$, the chain rule gives $\Delta_\Sigma u=\tr_{g_\Sigma}\Hess r-2H\langle n,\del_r\rangle$, because the traced normal second derivative of the immersion is $-2Hn$. All singular derivatives away from $Z_p$ are supported on $E_p$. The Cantor part of a derivative of a function of bounded variation vanishes on every set of finite one-dimensional Hausdorff measure in a two-dimensional chart~\cite[Chapter~3]{AFP}. The equation~\eqref{cutSigma} therefore eliminates this part on the surface.

At a jump point, the singular Hessian is $(\nabla_\Sigma u^+-\nabla_\Sigma u^-)\otimes N$ times length measure. Symmetry makes the jump parallel to $N$; semiconcavity gives its negative sign. Both gradient traces have norm at most $1$, since $u$ is $1$-Lipschitz. Hence the equation~\eqref{surface-cut-traces} holds with $0\leq\beta_p\leq2$. We extend $\beta_p$ by zero on $E_p\setminus J_p$. Near a point of $Z_p$, normal coordinates give $u=|x|+\Ocal(|x|^2)$ and $D(\nabla_\Sigma u)=\Ocal(1/|x|)$ off the origin, an integrable bound in dimension two. The flux of $\nabla_\Sigma u$ through a small circle is $\Ocal(\sigma)$, so its Laplacian has no atom there.
\end{proof}


\small

\begin{thebibliography}{Top00}
\bibitem{AFP}
\textsc{L.~Ambrosio, N.~Fusco, and D.~Pallara}, \emph{Functions of bounded variation and free discontinuity problems}, Oxford Mathematical Monographs, Clarendon Press, Oxford, 2000.


\bibitem{Hebda}
\textsc{J.J.~Hebda},
\emph{Metric structure of cut loci in surfaces and Ambrose's problem},
J. Differential Geom.~40 (1994), 621--642.

\bibitem{Kleiner}
\textsc{B.~Kleiner}, \emph{An isoperimetric comparison theorem}, Invent. Math. 108 (1992), 37--47.

\bibitem{Kusner}
\textsc{R.~Kusner}, \emph{Comparison surfaces for the Willmore problem}, Pacific J. Math. 138 (1989), 317--345.

\bibitem{LM}
\textsc{P. Laurain and A. Mondino,} 
\emph{Concentration of small Willmore spheres in Riemannian 3-manifolds}, 
Anal. PDE 7 (2014), 1901--1921.

\bibitem{LiYau:Willmore}
\textsc{P.~Li and S.-T.~Yau},
\emph{A new conformal invariant and its applications to the Willmore conjecture
and the first eigenvalue of compact surfaces}, Invent. Math. 69 (1982), 269--291.


\bibitem{MMU}
\textsc{C.~Mantegazza, G.~Mascellani, and G.~Uraltsev}, \emph{On the distributional Hessian of the distance function}, Pacific J. Math. 270 (2014), 151--166.

\bibitem{Meyer}
\textsc{W.T.~Meyer},
\emph{Toponogov's theorem and applications}, unpublished lecture notes, College on Differential Geometry, Trieste, 1989.

\bibitem{Ozols}
\textsc{V.~Ozols},
\emph{Cut loci in Riemannian manifolds},
T\^ohoku Math. J. 26 (1974), 219--227.

\bibitem{Petersen}
\textsc{P.~Petersen},
\emph{Riemannian geometry},
Graduate Texts in Mathematics, Vol.~171, second edition, Springer, New York, 2006.

\bibitem{Ritore}
\textsc{M.~Ritor\'e}, \emph{Optimal isoperimetric inequalities for three-dimensional Cartan--Hadamard manifolds}, Global theory of minimal surfaces, Clay Math. Proc., Vol.~2, Amer. Math. Soc., Providence, RI, 2005, pp.~395--404.

\bibitem{Scharrer}
\textsc{C.~Scharrer}, \emph{Some geometric inequalities for varifolds on Riemannian manifolds based on monotonicity identities}, Ann. Global Anal. Geom. 61 (2022), 691--719.

\bibitem{Schulze}
\textsc{F.~Schulze}, \emph{Nonlinear evolution by mean curvature and isoperimetric inequalities}, J. Differential Geom. 79 (2008), 197--241.

\bibitem{Simon}
\textsc{L.~Simon},
\emph{Existence of surfaces minimizing the Willmore functional},
Comm. Anal. Geom. 1 (1993), 281--326.

\bibitem{Topping:MCF}
\textsc{P.~Topping},
\emph{Mean curvature flow and geometric inequalities},
J. Reine Angew. Math. 503 (1998), 47--61.

\bibitem{Topping:W}
\textsc{P.~Topping},
\emph{Towards the Willmore conjecture,}
Calc. Var. 11 (2000), 361--393.
 

\end{thebibliography}


\normalsize
\end{document}